\documentclass[12pt]{article}
\usepackage{amsmath,amsthm,amssymb,latexsym,amsfonts,epsfig, color, authblk, geometry, tabularx, caption}
\usepackage{tikz}
\usepackage{xfrac}
\usepackage{pgfplots}
\usetikzlibrary{spy}
\usetikzlibrary{backgrounds}
\usetikzlibrary{decorations}
\pgfdeclarelayer{background layer}
\pgfsetlayers{background layer,main}
\definecolor{darkgray}{rgb}{0.25,0.25,0.25}
\definecolor{lightgray}{rgb}{0.75,0.75,0.75}
\usepackage{enumerate}

\newcommand{\QM}{\mathcal{M}}

\newcommand{\dlim}{\displaystyle \lim}
\newcommand{\B}{{\mathcal B}}

\newcommand{\ns}[1]{\overline{#1}} 
\newcommand{\tr}{^\intercal}
\newcommand{\matr}[1]{\begin{bmatrix} #1 \end{bmatrix}}    
\newcommand{\rec}[1]{{#1}^\infty}

\renewcommand{\P}{{\mathcal P}}

\newcommand\N{{\mathbb N}}
\newcommand\R{{\mathbb R}}
\newcommand{\LAA}{Linear Algebra and its Applications}

\newcommand{\refer}[1]{\mbox{\rm(\ref{#1})}}
\newtheorem{theorem}{Theorem}
\newtheorem{lemma}{Lemma}
\newtheorem{proposition}{Proposition}
\newtheorem{corollary}{Corollary}
\newtheorem{definition}{Definition}

\newtheorem{example}{Example}

\newcommand{\brac}[1]{{\left(#1\right)}}

\renewcommand{\int}{\operatorname{int}}
\newcommand{\closure}{\operatorname{cl}}
\newcommand{\socp}{\mathcal{L}}

\usepackage[numbers,sort]{natbib}
 \def\newblock{\ }%
 \bibpunct[, ]{[}{]}{,}{n}{}{,}%

\usepackage[colorlinks=true,breaklinks=true,bookmarks=true,urlcolor=blue,
     citecolor=blue,linkcolor=blue,bookmarksopen=false,draft=false]{hyperref}

\title{\bf Positive polynomials on unbounded domains}

\author[1]{Javier Pe\~na\thanks{{\tt jfp@andrew.cmu.edu} } }
\author[2]{Juan ~C. Vera\thanks{\tt j.c.veralizcano@uvt.nl}}
\author[3]{Luis ~F. Zuluaga\thanks{ {\tt luis.zuluaga@lehigh.edu} (Corresponding author)}}

\affil[1]{Tepper School of Business, Carnegie Mellon University, Pittsburgh, PA, USA 15213}
\affil[2]{Tilburg School of Economics and Management, Tilburg University, Tilburg, The Netherlands}
\affil[3]{Department of Industrial and Systems Administration, Lehigh University, Bethlehem, PA, USA 18015}

\date{}

\begin{document}

\maketitle
\abstract{
Certificates of non-negativity such as Putinar's
Positivstellensatz have been used to obtain powerful numerical techniques
to solve polynomial optimization (PO) problems. Putinar's certificate uses sum-of-squares (sos) polynomials to certify the non-negativity of a given polynomial over a domain defined by polynomial inequalities.  This  certificate assumes the Archimedean property of the associated quadratic module, which in particular implies compactness of the domain.
In this paper we characterize the existence of a certificate of non-negativity for polynomials over a possibly unbounded domain, without the use of the associated quadratic module.
Next, we show that the certificate can be used to
convergent linear matrix inequality (LMI) hierarchies for PO problems with unbounded feasible sets. Furthermore, by using copositive polynomials to certify non-negativity, instead of sos polynomials,
the certificate allows the use of a very rich class of convergent LMI hierarchies
to approximate the solution of general PO problems.
Throughout the article we illustrate our results with various examples
certifying the non-negativity of  polynomials over possibly unbounded
sets defined by polynomial equalities or inequalities.}


\section{Introduction}
\label{sec:intro}

A certificate of non-negativity is generally understood as an expression that makes the non-negativity of the function in question evident. Certificates of non-negativity are fundamental tools in optimization,
and underlie powerful algorithmic techniques for various types of optimization problems.
For example, Farkas Lemma \cite[see, e.g.,][]{Chva83} can be interpreted as an expression that makes the non-negativity of a linear
polynomial over a polyhedron evident. Similarly, the $S$-Lemma
\cite[see, e.g.,][]{BenTN01} can be used to certify whether a quadratic
polynomial is non-negative over a set with non-empty interior defined by a single quadratic
inequality.
More elaborate certificates of non-negativity for higher degree polynomials over a basic semialgebraic set include the classical P\'olya's Positivstellensatz~\cite{HardLP88}, and the more modern Schm\"udgen's Positivstellensatz~\cite{Schm91} and  Putinar's Positivstellensatz~\cite{Puti93}
(herein, we will use the terms Positivstellensat and certificate of non-negativity
intercheangably).
These certificates of non-negativity for polynomials can be used to solve polynomial optimization~(PO)
problems; that is, optimization problems in which both the objective and the constraints are polynomials
on the decision variables. Clearly, PO problems encompass a very general class of optimization problems including
combinatorial and non-convex optimization problems.

More specifically, certificates of non-negativity for polynomials can be used
to construct a hierarchy of {\em linear matrix inequality} (LMI)  problems (i.e.,
optimization problems with a linear objective and a LMI constraint) \cite[cf.,][]{BlekPT13} whose
objective converges to the objective of the PO problem of interest.
In turn, {\em interior-point methods} \cite[cf.,][]{Rene01} can be effectively used to obtain the global solution
of the LMI problems.
As illustrated in recent research work \cite[e.g.,][among numerous others]{deKlP02, Jibed05, PenaVZ05, Lass09, HenrL03, Lass01, PapaPP02, Parr03}, P\'olya's, Schm\"udgen's, and Putinar's Positivstellensatz
are suitable certificates of non-negativity for this purpose.

In general, the non-negativity of a polynomial can be certified using the Krivine-Stengle's Positivstellensatz \citep{Kriv64,Sten74}.
However, as discussed in detail by \citet{JeyaLL13}, this certificate is not readily amenable to the LMI solution approach  for PO problems outlined above.  On the other hand, Schm\"udgen's and Putinar's Positivstellensatz
are amenable to this LMI solution approach, but require
the underlying semialgebraic set $S$ (over which the non-negativity of the polynomial is to be certified) to be bounded. More precisely, a sufficient condition for Schm\"udgen's Positivstellensatz
to hold is that the set $S$ must be compact. In the case of Putinar's Positivstellensatz, a sufficient condition for the certificate to hold is that the {\em quadratic module}~\citep[cf.,][]{PresD01}
generated by the polynomials defining the set $S$ must be {\em Archimedean}~\cite[cf.,][]{Puti93}, which implies that the set~$S$ is compact.

Given the key role that compactness and the Archimedean property play in
both Schm\"ud\-gen's and Putinar's certificates of non-negativity, there has been active and
relevant research into studying cases in which certificates of non-negativity
can be obtained and applied without relying on those properties.
Global optimization (i.e., unrestricted optimization) of a polynomial is a
general case in this class that has received a lot of attention in the literature.
For example, consider the work of \citet[][among others]{Pham2008, Schweighofer2006}. Here, we focus on the case of optimizing a polynomial over constrained non-compact feasible sets.
Recent examples of work in this direction are the results of \citet{Powe04, Mars10}, who
derived certificates of non-negativity
for the case in which the underlying domain is
a cylinder with a compact cross-section. For more general settings,
\citet{NieDS06, DemmNP07, Mars09} provide
certificates of non-negativity that do not require the underlying set to
be compact
via gradient and KKT ideals. More recently,  \citet{JeyaLL13}
provide a certificate of non-negativity for non-compact semi-algebraic sets when the quadratic
module of the polynomials defining the semi-algebraic feasible set of a PO problem,
 together with a polynomial defining
a level set of the objective's  PO problem is Archimedean.
As mentioned before, the classical $S$-Lemma can be viewed as a certificate of non-negativity for quadratic polynomials.
Thus, extensions of the $S$-Lemma such as the ones
by \citet[][]{SturZ01, xia15, wang15} also belong to this direction of research.

We present a new certificate of non-negativity for polynomials over the possibly unbounded set obtained from the intersection of a closed domain $S$ and $h^{-1}(0) = \{x \in \R^n: h(x) = 0\}$, the zero set of a given polynomial~$h(x)$.  It is evident that if $p(x)$ is non-negative on the domain $S$, then $p(x) + h(x) q(x)$ is non-negative on the domain $S \cap h^{-1}(0)$ for any polynomial $q(x)$.
In \cite[][Theorem~1]{PenaVZ08} it is shown that (modulo an appropriate closure) the converse is true
when the domain $S$ is compact, thereby establishing a certificate of non-negativity for polynomials on $S \cap h^{-1}(0)$ in terms of non-negative polynomials on $S$.

The main contribution of this article is to show that
under suitable conditions on $h(x)$ and~$S$, the
non-negativity of a polynomial over the set $S \cap h^{-1}(0)$
can be certified in terms of the non-negative polynomials on $S$
even if the set $S$ is unbounded. Moreover, a characterization of
the sets  $S \cap h^{-1}(0)$ for which the certificate of non-negativity
exists is provided in terms of an appropriate condition on the
set $S \cap h^{-1}(0)$ (cf., Theorem~\ref{main.thm.equ}).
This result is obtained via a key characterization of
the closure of the non-negative polynomials on $S$ plus the
polynomials generated by the ideal of $h(x)$ (cf., Theorem~\ref{thm:finerEqGen}).
Beyond providing a novel contribution to the literature on certificates of
non-negativity over possibly unbounded sets, Theorem~\ref{main.thm.equ} has some
important consequences.

Unlike
previous related results in \cite{DemmNP07,Mars09,NieDS06, JeyaLL13},
the proposed certificate of non-negativity is
independent of the polynomial defining the objective of an associated PO problem.
Instead, the certificate is
written purely in terms of the set of non-negative polynomials over a set $S$ and the ideal generated by $h(x)$.
Also, unlike the recent related results in  \citet{JeyaLL13},
and as a result of the use of the non-negative polynomials on a set $S$, the associated quadratic module and the Archimedean property
are not used to characterize the cases in which the proposed certificate of non-negativity holds.

The certificate of non-negativity presented here readily allows the
use of {\em copositive polynomials}~\cite[cf.,][]{BlekPT13} to certify
the non-negativity of a polynomial (as opposed to the more common use
of sums-of-squares polynomials to certify non-negativity). As a consequence,
a very rich class of convergent hierarchies of LMI problems to approximate the solution of general PO problems is obtained. The fact that copositive polynomials
can be used in the proposed certificate of non-negativity, together
with Polya's Positivstellensatz~\cite[see, e.g.,][]{HardLP88}, means that
convergent linear programming~(LP) hierarchies can be constructed to approximate
the solution of general PO problems.

Moreover, the results presented in the article
provide an interesting bridge between the results on
certificates of non-negativity in algebraic geometry (like
 Schm\"udgen's and Putinar's Positivstellensatz) and
 results on certificates of non-negativity arising in
 the general area of {\em quadratic programming}~\cite[cf.,][]{BenTN01}.

The remainder of the article is organized as follows.
In Section~\ref{sec:main.result}, we  motivate and formally state our main result; namely,
a characterization for the existence of a certificates of non-negativity for polynomials
over a possibly unbounded set
obtained from the intersection of a closed
domain $S$ and the zero set of a given polynomial $h(x)$ (cf., Theorem~\ref{main.thm.equ}).
Also, we present some important consequences of this result.
In Section~\ref{sec:inequalities}, we extend the results presented
in Section~\ref{sec:main.result} to consider inequality constraints; that is, when the underlying set of interest
is defined as the intersection of a closed domain~$S$ and the set $h^{-1}(\R_+):=\{x \in \R^n: h(x) \ge 0\}$ for
a given polynomial $h(x)$ (cf., Theorem~\ref{thm:ineq}). In Section~\ref{sec:extensions}, we provide further relevant extensions of our main results.
For the purpose of clarity, the presentation of some of the proofs is deferred until
Sections~\ref{sec:proofs},~\ref{sec:counter}, and~\ref{sec:infconds}. In Section~\ref{sec:remarks}, we provide some concluding remarks.


\section{A new certificate of non-negativity.}\label{sec:main.result}

Certificates of non-negativity for polynomials such as the classical P\'olya's Positivstellensatz~\cite{HardLP88}, and the more modern Schm\"udgen's \cite{Schm91} and  Putinar's  Positivstellensatz~\cite{Puti93},
are central in
polynomial optimization  \citep[cf.,][]{Anjo12,BlekPT13,Parr00,Lass01, Marshall2008}. Both Schm\"udgen's and  Putinar's Positivstellensatz certify the non-negativity
of a polynomial using {\em sum of squares} (sos) polynomials which can be defined as follows.
\begin{definition}[Sums of squares (sos) polynomials]
\label{def:sos}
 Let $\Sigma$ be the set of sos polynomials; that is,
 $\Sigma = \{\sigma \in \R[x]: \sigma(x) = \sum_{i=1}^r q_i(x)^2
\text{ for some }r \in \N, q_i(x) \in \R[x], i=1,\dots,r\}$. Also, let~$\Sigma_d$ be the set of sos polynomials of degree $\le d$;
that is, $\Sigma_d= \Sigma \cap \R_d[x]$.
\end{definition}

Above,
$\R[x]:=\R[x_1,\dots,x_n]$ denotes the set of real ($n$-variate) polynomials, and
$\R_d[x]:=\R_d[x_1,\dots,x_n]$ denotes the set of real  polynomials
of degree $\le d$. Clearly, sos polynomials are always non-negative on any domain. In general we will refer
to non-negative polynomials on a given domain using the following notation.

\begin{definition}[Cone of non-negative polynomials]
For any set $S \subseteq \R^n$ let $\P(S)$ be set of non-negative ($n$-variate) polynomials
on $S$; that is, $\P(S) = \{p \in \R[x]: p(x) \ge 0 \text{ for all } x \in S\}$. Also, let
$\P_d(S)$ be the set of non-negative polynomials of degree $\le d$ on $S$;
that is, $\P_d(S) = \P(S) \cap \R_d[x]$.
\end{definition}

Notice that while $\R[x]$ is an infinite dimensional vector space, $\R_d[x]$ is finite dimensional. In particular each polynomial in $\R_d[x]$ can be identified with the vector of its coefficients in 
$\R^{n+d \choose d}$; notions such as limits and closures over sets of polynomials in $\R_d[x]$ are naturally interpreted through this identification.

\citet{Puti93}
uses sos polynomials to
 certify the non-negativity of polynomials over a given basic semi-algebraic set  $S = \{x \in \R^n: g_i(x) \ge 0, i=1,\dots,m\}$. More precisely,
Putinar's Positivstellensatz provides a certificate of non-negativity for polynomials over a given basic semi-algebraic set  $S = \{x \in \R^n: g_i(x) \ge 0, i=1,\dots,m\}$ in terms of the {\em quadratic module} $\QM(g_1,\dots,g_m)$ generated by
the polynomials $g_1,\dots,g_m \in \R[x]$ when $\QM(g_1,\dots,g_m)$ is {\em Archimedean}~\citep[cf.,][]{PresD01}.
A neccesary condition for the Archimedean property to hold is that the set $S$ is compact.

Here, we present novel certificates of non-negativity
for polynomials over potentially unbounded sets. The results are based on
extending the
certificate of non-negativity derived by \citet[][Corollary~2]{PenaVZ08}
for polynomials on the intersection
of a compact set $S$ and the zero set
\[h^{-1}(0) := \{x \in \R^n: h(x) = 0\},\]
of a polynomial $h \in \R[x]$. Specifically, consider the following proposition that is a direct consequence
of~\cite[][Corollary~2]{PenaVZ08}.

\begin{proposition}
\label{th:compact} Assume $S \subseteq \R^n$ is compact and $h \in \P_d(S)$. Then
\begin{equation}
\label{eq.result_compact}
\P_d(S\cap h^{-1}(0)) = \closure(\P_d(S) + h(x)\R_{d-\deg(h)}[x]).
\end{equation}
\end{proposition}

Notice that one of the
inclusions in~\eqref{eq.result_compact} readily follows
even if the set $S$ is not necessarily compact.

\begin{lemma}\label{lem:subset}
Let $S \subseteq \R^n$, $h \in \P_d(S)$. Then $\closure(\P_d(S) + h(x)\R_{d-\deg(h)}[x]) \subseteq \P_d(S\cap h^{-1}(0))$.
\end{lemma}
That is, being an element of the $\closure(\P_d(S) + h(x)\R_{d-\deg(h)}[x])$ certifies the non-negativity
of a polynomial on $S \cap h^{-1}(0)$, provided that the polynomial $h \in \P_d(S)$.
In light of Lemma~\ref{lem:subset}, Proposition~\ref{th:compact} states that if $S$ is compact, then the certificate
of non-negativity~\eqref{eq.result_compact} exists for
any polynomial non-negative on $S \cap h^{-1}(0)$ when $h \in \P_d(S)$.
As the next example illustrates, this fact might fail when the set $S$ is unbounded.

\begin{example}\label{ex:couterInit}
As illustrated in Figure~\ref{fig:counterexample},
let $d=4$, $S = \R^2_+$, and $h(x) = (x_1x_2 + 1)(x_1-x_2)^2$. Note that $h(x) \in \P_4(\R^2_+)$ and  $q(x) := x_2^4 - x_1^4 \in \P_4(\R^2_+ \cap h^{-1}(0))$.
 Let $x^k = (k, -1/k)$ and notice that $h(x^k) = 0$ for all $k > 0$.
 We claim that $\lim_{k \to \infty} k^{-4}p(x^k)\ge 0$
 for any $p(x) \in \closure(\P_4(\R^2_+) + h(x)\R)$, and therefore $q(x) \notin \closure(\P_4(\R^2_+) + h(x)\R)$, as
 $\lim_{k \to \infty} k^{-4}q(x^k)= -1$.
To prove the claim, let $p(x) = \lim_{i \to \infty} p_i(x) + c_ih(x)$ with $p_i(x) \in \P_4(\R^2_+)$ and $c_i \in \R$ for all~$i \ge 0$. Also, for any $k > 0$ let $\tilde x^k = (k, 1/k)$. For any $i\ge0$, we have that $k^4(p_i(x^k) - p_i(\tilde x^k))$ is a polynomial of degree 7
(on $k$). Thus, using that $\tilde x^k \in \R^2_+$, for any $i \ge 0$ and $k>0$  we have
\[
\lim_{k \to \infty} \frac{p_i(x^k)}{k^4} = \lim_{k \to \infty} \frac{k^4(p_i(x^k)-p_i(\tilde x^k))}{k^8} + \lim_{k \to \infty}\frac{p_i(\tilde x^k)}{k^4} = \lim_{k \to \infty}\frac{p_i(\tilde x^k)}{k^4} \ge 0.
\]
Therefore,

{\tt TO BE FIXED... }

\[
\lim_{k \to \infty} \frac{p(x^k)}{k^4} = 
\lim_{k \to \infty} \frac{k^4(p(x^k)-p(\tilde x^k))}{k^8} + \frac{p(\tilde x^k)}{k^4} =
\lim_{k \to \infty} \lim_{i \to \infty} \frac{p_i(\tilde x^k) + 2(k-\frac{1}{k})^2c_i}{k^4}
\]

\end{example}

 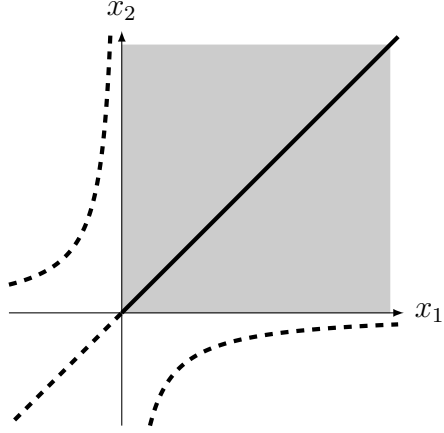
\begin{figure}
\begin{center}
\begin{tikzpicture}[scale=.75]
   \draw[fill = gray!40!white, gray!40!white] (0,0) rectangle (4.75,4.75);
    \draw[-latex,color=black] (-2,0) -- (5,0);
  \draw[-latex,color=black] (0,-2) -- (0,5);

 \node[right] at (5,0){$x_1$};
 \node[above] at (0,5){$x_2$};
   \draw[smooth,samples=100,domain=-1.9:0, ultra thick, color = black, dashed]
  plot(\x,\x);
  \draw[smooth,samples=100,domain=-0:4.9, ultra thick, color = black]
  plot(\x,\x);
   \draw[smooth,samples=100,domain=0.5:4.9, ultra thick, color = black, dashed]
  plot(\x,-1/\x);
  \draw[smooth,samples=100,domain=-2:-0.2, ultra thick, color = black, dashed]
  plot(\x,-1/\x);
\end{tikzpicture}
\end{center}
\caption{Illustration of Example~\ref{ex:couterInit}. The set $S = \R^2_+$ is illustrated in lightgray, while the bold line
illustrates the set
$S \cap h^{-1}(0)$ with $h^{-1}(0) = \{x \in \R^2: (x_1x_2 + 1)(x_1-x_2)^2 = 0\}$. The union of the
bold line and the doted line illustrate the set  $h^{-1}(0)$.}
\label{fig:counterexample}
\end{figure}

The key point in Example~\ref{ex:couterInit} is the fact that $x^k$, $k >0$, which is a sequence of zeroes of~$h$, gets arbitrarily close to the unbounded
set $S = \R^2_+$ when $k \to \infty$; that is, $h$ has a {\em zero at infinity} in $S$ \cite[cf.,][]{Rezn00}.
This is illustrated in Figure~\ref{fig:counterexample}.
 The special property of this zero at infinity is that it is not a zero at infinity of
 the set $S \cap h^{-1}(0)$ (illustrated with a bold line in Figure~\ref{fig:counterexample}).
In Theorem~\ref{main.thm.equ} below, we show that there
is an equivalence relationship between
the existence of a certificate of non-negativity of the form
\eqref{eq.result_compact} and an appropriate condition on the set $S \cap h^{-1}(0)$.
When $S$ is compact, the condition is evidently satisfied as
neither the set
$S$ or $S \cap h^{-1}(0)$ has zeros at infinity. As shown in
Example~\ref{ex:cont} this is not necessarily the case
 when the set $S$ is unbounded.
To formalize this condition, we first introduce some additional definitions and results.

Given a polynomial $h \in \R[x]$, let $\tilde h(x)$ denote the homogeneous component of $h$ of highest total degree.  In other words,
$\tilde h(x)$ is obtained by dropping from $h$ the terms whose total degree is less than $\deg(h)$.

\begin{definition}[Horizon cone]
\label{def:horizon}
The horizon cone $\rec{S}$  of
a given set  $S \subseteq \R^n$ is defined as (see, e.g., \cite{RockW98}):
\[
\rec{S}:=\{ y \in \R^n:  \text{ there exist } x^k \in S, \; \lambda^k \in \R_+, \; k=1,2,\dots \; \text{ such that }\lambda^k \downarrow 0 \text{ and }   \lambda^k x^k \rightarrow y\}.
\]
In particular, if $S = \emptyset$, define $\rec{S} := \emptyset$.
\end{definition}

As Proposition~\ref{prop:inf} below shows, the non-negativity of a polynomial and the
non-negativity of its homogenous component of highest degree can be related through
the horizon cone.

\begin{proposition}[{\citep[][]{PenaVZ15}}]
\label{prop:inf} Let $S \subseteq \R^n$. If $p(x) \in \P_d(S)$, then $\tilde p(x) \in \P_d(\rec{S})$.
\end{proposition}


Next, we provide a general and interesting characterization of $\closure(\P_d(S) +h(x)\R_{d-\deg(h)}[x])$
(cf., eq.~\eqref{eq.result_compact}) that is essential for
the results presented here. The proof of this characterization
is presented in Section~\ref{sec:proofs} for the purpose of clarity.

\begin{theorem}\label{thm:finerEqGen}
Assume $K\subseteq \R^n$ is a pointed closed convex cone. Let $S\subseteq K$ be a closed set and $h \in \P_d(S)$.
Then,
\begin{equation}
\label{eq:finerGen}
\closure(\P_d(S) + h(x)\R_{d-\deg(h)}[x])= \{p(x) \in \P_d(S\cap h^{-1}(0)): \tilde p(x) \in \P_d( \rec{S} \cap \tilde h^{-1}(0))\}.
\end{equation}
\end{theorem}

Notice that if $p(x)$ bounded below on $S$, Proposition~\ref{prop:inf} implies $\tilde p(x) \in \P_d( \rec{S}) \subseteq \P_d( \rec{S} \cap \tilde h^{-1}(0))$. Thus, for polynomials bounded below on $S$,
\begin{equation}\label{eq:result}
p(x) \in \P_d(S\cap h^{-1}(0)) \text{ if and only if }p(x) \in \closure(\P_d(S) + h(x)\R_{d-\deg(h)}[x]).
\end{equation}
In particular, if $p(x)$ is coercive, i.e. if $\|x\| \rightarrow \infty$ implies $p(x) \rightarrow \infty$, we have that \eqref{eq:result} follows \citep[cf.,][]{JeyaLL13}.
We are interested on which conditions on $S$ and $h$ would imply that \eqref{eq:result} holds for all polynomials.
From Theorem~\ref{thm:finerEqGen}, and Proposition~\ref{prop:inf},  it follows that the condition $\rec{(S \cap h^{-1}(0))} = \rec{S} \cap \tilde h^{-1}(0)$ implies $\P_d(S\cap h^{-1}(0)) = \closure(\P_d(S) + h(x)\R_{d-\deg(h)}[x])$.
Theorem~\ref{main.thm.equ} below shows that this condition is not only sufficient but also necessary to obtain
the desired characterization of non-negative certificates for polynomials over potentially unbounded sets.

\begin{theorem}\label{main.thm.equ}
Assume $K \subseteq \R^n$ is a closed convex pointed cone.  Let $S\subseteq K$ be a closed set and $h\in \P_d(S)$
with $d \ge 2$. Then
\begin{equation}\label{equal.identity}
\P_d(S\cap h^{-1}(0)) = \closure(\P_d(S) + h(x)\R_{d-\deg(h)}[x]),
\end{equation}
if and only if
\begin{equation}\label{eq.condition}
\rec{(S \cap h^{-1}(0))} = \rec{S} \cap \tilde h^{-1}(0).
\end{equation}
\end{theorem}
\begin{proof}
To show that~\eqref{eq.condition} implies \eqref{equal.identity}, first notice that
from Theorem~\ref{thm:finerEqGen} we have
\begin{align*}
&\closure(\P_d(S) + h(x)\R_{d-\deg(h)}[x]) & \\
&= \{p(x) \in \P_d(S\cap h^{-1}(0)): \tilde p(x) \in \P_d( \rec{S} \cap \tilde h^{-1}(0))\} & \\
& = \{p(x) \in \P_d(S\cap h^{-1}(0)): \tilde p(x) \in \P_d(\rec{(S \cap h^{-1}(0))})\} & \text{(by \eqref{eq.condition})}\\
& = \P_d(S\cap h^{-1}(0)) & \text{(by Proposition~\ref{prop:inf})}.
\end{align*}
The fact that \eqref{equal.identity} implies~\eqref{eq.condition} follows from the counterexample detailed in Section~\ref{sec:counter}, which is a generalized version of Example~\ref{ex:couterInit}.
\end{proof}

\begin{example}[Example~\ref{ex:couterInit} revisited]
\label{ex:cont}
Note that condition \eqref{eq.condition} is not satisfied in the case of
$S= \R^2_+, h(x) = (x_1x_2+1)(x_1-x_2)^2$ considered in Example~\ref{ex:couterInit}. Specifically,
$\rec{(S\cap h^{-1}(0))} = \{(t,t): t \ge 0\} \subsetneq \rec{S} \cap \widetilde{h}^{-1}(0) = \{ (t,t), (0,t), (t,0): t \ge 0\}$.
\end{example}

As shown in Section~\ref{sec:extensions}, the conditions of Theorem~\ref{main.thm.equ} can be relaxed as long as either the number of variables or the degree of the polynomials involved in the certificate are appropriately increased.

Note that unlike related non-negativity certificates, in which an appropriate quadratic module (or {\em pre-order}~\cite[cf.,][]{Schm91}) is
used to certify the non-negativity of a polynomial, the non-negativity certificate in Theorem~\ref{main.thm.equ}
exploits the structure of the underlying set over which the non-negativity
of a polynomial is desired to be certified by allowing to extend any non-negativity certificate over the set $S$ to the set $S \cap h^{-1}(0)$.
Also,
the certificate of non-negativity is recursive in nature. That is, it can
be applied repeatedly to address the complexity of certifying the
non-negativity of a polynomial over a given semi-algebraic set.
This fact will be used in Section~\ref{sec:balls} to derive new
interesting LMI hierarchies for PO problems.

Also, notice that unlike certificates of non-negativity based on the use
of a quadratic module (or pre-order), the degree of
the polynomials used to certify non-negativity is known a priori, and
equal to the degree of the polynomial whose
non-negativity is to be certified.
This behaviour
of the degree of the polynomials involved in the proposed
certificate of non-negativity
have
proved to be useful in the related work~\cite{PenaVZ08, GhadAV11, PenaVZ15},
when the underlying set is compact. Here, this fact will be key to
derive some of the relevant consequences of the non-negative certificate in Theorem~\ref{main.thm.equ} presented therein.

It is important to mention that \citet[][Thm. 4.2]{putinar1999} use conditions over the natural homogenization of relevant polynomials to obtain certificates of non-negativity of a polynomial over general semialgebraic sets. Although homogenization will play a key role in deriving Theorem~\ref{thm:finerEqGen} and~\ref{main.thm.equ}, the conditions under which these results hold are characterized (among others) by the homogeneous component of largest degree of a relevant polynomial. Moreover, unlike \cite[Thm. 4.2]{putinar1999}, Theorem~\ref{main.thm.equ} leads to a rich class of different classes of LMI hierarchies that can be used to certify the non-negativity of a polynomial.

\subsection{Certifying non-negativity in a non-Archimedean case}
\label{sec:Lasserre}
One of the key properties of the certificate of non-negativity provided in
Theorem~\ref{main.thm.equ}
is that the Archimedean property is not used to characterize the cases in which
the non-negativity certificate can be applied. Recall that the Archimedean
property~\citep[cf.,][]{PresD01}
is used in Putinar's Positivstellensatz~\citep[cf.,][]{Puti93} to certify the non-negativity of a polynomial
on compact semialgebraic sets. Recently,  \citet[][]{JeyaLL13} obtained a
certificate of non-negativity of polynomials over possibly non-compact
semialgebraic that holds whenever a suitable quadratic module is Archimedean.
 This result can be
used to obtain hierarchy of LMI approximations for general PO problems
with a possibly unbounded feasible set.
As Proposition \ref{prop:prelasserre} below illustrates, the certificate of non-negative stated in
Theorem~\ref{main.thm.equ} can be used to obtain an LMI reformulation
for PO problems in which the Archimedean property
is not required.

First we state the following lemma, which is used in the proof of Proposition~\ref{prop:prelasserre}.
\begin{lemma}\label{lem:lintransf} Let $A \in \R^{n\times n}$ be an invertible matrix. Let $b \in \R^n$. Let $S \subset \R^n$ be given. Then for any $p(x) \in \R[x]$ and $d>0$, $p(x) \in \P_d(S)$ if and only if $p(A^{-1}(x-b)) \in P_d(AS + b)$.
\end{lemma}

\begin{proposition}
\label{prop:prelasserre}
Let $q(x)$ be a concave quadratic, and let $h(x)$ be a non-negative quadratic such that
$U = \{x \in \R^n: q(x) \ge 0,  h(x) = 0\} \neq \emptyset.$
If $\{x \in \R^n: q(x) \ge 0\}$ contains no lines, then
\[
\P_2(U) = \closure\left\{\sigma(x) + \lambda q(x) + \mu h(x) : \lambda \in \R_+,
\mu \in \R, \sigma \in \Sigma_2\right\}.
\]
\end{proposition}
\begin{proof}
As $q(x)$ is concave, write $q(x) = c_0 + c\tr x - x\tr C x$ where $C \succeq 0$.  Assume $C$ is of rank $m$, and write $C = Q\tr \Lambda Q$, where $Q$ is a orthogonal matrix and $\Lambda$ is the diagonal matrix of eigenvalues of $C$, such that $\Lambda_{1,1} \ge \cdots \ge \Lambda_{m,m} > 0 = \Lambda_{m+1,m+1} = \cdots = \Lambda_{nn}.$ From Lemma~\ref{lem:lintransf}, it is enough to prove the statement for the pair $q(Q\tr x)$ and $h(Q\tr x)$ i.e. we can assume
\[q(x) = a + b\tr x - x\tr \Lambda x = a + \sum_{i=1}^{m} \frac{b_i^2}{4\Lambda_{ii}} + \sum_{i=m+1}^n b_i x_i - \sum_{i=1}^{m} \left( \frac{b_i}{2\Lambda_{ii}^{1/2}} - \Lambda_{ii}^{1/2}x_i \right)^2.\]
Let $m_1$ be the number of non-zeros in the vector $(b_{m+1},\dots,b_n)$.
 Using Lemma~\ref{lem:lintransf} again
 \[q(x) = a  + \sum_{i=m+1}^{m+m_1} x_i - \sum_{i=1}^{m} x_i^2.\]
If $m_1>1$, we use Lemma~\ref{lem:lintransf} with the affine transformation
 \[x \mapsto \left(x_1, \dots, x_m , a +1/4 + \sum_{i=m+1}^{m+m_1} x_i , x_{m+2},\dots,x_n \right),\]
to obtain that in this case we can assume $m_1=1$ and $a = -1/4$.

As $S = \{x \in \R^n: q(x) \ge 0\}$ contains no lines, and lines are send to lines by affine transformations, we obtain $n = m + m_1$.

Now we claim,
\begin{equation}\label{eq:ladesiempre}
\P_2(U) = \closure (\P_2(S) + h(x)\R),
\end{equation}
and since $S$ is non-empty, from the regularized $S$-Lemma \citet{SturZ01} we have that $\P_2(S) =  \{\sigma(x) + \lambda q(x):
\lambda \in \R_+, \sigma \in \Sigma_2\}$. The result follows after replacing this characterization of $\P_2(S)$ in~\eqref{eq:ladesiempre}.

To show the claim, notice that if $m_1$ = 0, $S$ is compact and \eqref{eq:ladesiempre} follows from Theorem~\ref{th:compact}. Thus, assume $m_1 = 1$, i.e. $n = m+1$. Notice that $x_{m+1} \ge 1/4 +\|x\|^2$ implies $x_{m+1}^2 = (x_{m+1} - 1/2)^2 + x_{m+1} - 1/4 \ge \|x\|^2$ and $x_{m+1} \ge 0$, that is (as illustrated in Figure~\ref{fig:sturm}) $S \subseteq \socp$, where $\socp = \{ (x,x_{m+1}) \in \R^{m+1}: x_{m+1} \ge \|x\|_2\}$  is the {\em second-order (Lorentz) cone}
\cite[cf.,][]{LoboVBL98}, which is a pointed cone.
By assumption, $h(x) \in P_2(\R^n) \subset \P_2(S)$. Also, condition~\eqref{eq.condition} is satisfied as shown in Proposition~\ref{prop:infconds} in Section~\ref{sec:infconds}. Hence  Theorem~\ref{main.thm.equ} yields \eqref{eq:ladesiempre}.
\end{proof}

Now consider the problem of certifying the non-negativity of the polynomial $p(x,x_{m+1}) = x_{m+1}\sum_{i=1}^m x_i$ on $U = \{(x,x_{m+1}) \in \R^{m+1}: x_{m+1} \ge \frac{1}{4} + \sum_{i=1}^m x_i^2,  \sum_{i=1}^m x_i^2 = 0\}.$
As Illustrated in Figure~\ref{fig:lasserre}, because
the set $U$ is unbounded a certificate of non-negativity for $p(x,x_m)$ on~$U$ based on Putinar's Positivstellensatz~\citep[cf.,][]{Puti93} is not
guaranteed to exist, as the fact that $U$ is unbounded implies that its underlying associated quadratic module is not Archimedean.
Also, as illustrated in Figure~\ref{fig:lasserre}, a certificate of non-negativity for $p(x,x_{m+1})$ on~$U$ based on \citet[][Corollary 3.1]{JeyaLL13}
is not guaranteed to exist. To see this, note that~$c$ in \citet[][Corollary 3.1]{JeyaLL13} must satisfy $c > p(x',x'_{m+1})$ for some $(x',x'_{m+1}) \in U$, thus $c >0$. But for any $c > 0$, the set
$U \cap \{(x_0, x) \in \R^{n+1}: c > p(x_0, x)\} = \{(x_0, x) \in \R^{n+1}: x_0 > \frac{1}{2}\}$,
is unbounded. Thus, the quadratic module of $(x_{m+1} - \frac{1}{2}) - \sum_{i=1}^n x_i^2$,
$\sum_{i=1}^n x_i^2$, $c - x_{m+1}\sum_{i=1}^n x_i$ is not Archimedean for any $c > 0$. This is illustrated in Figure~\ref{fig:lasserre}. In contrast, the non-negativity for $p(x,x_{m+1})$ on $U$
can be certified by a LMI formulation using Proposition~\ref{prop:prelasserre}.

\begin{figure}[!htb]
\begin{center}
\begin{tikzpicture}[scale=.75]

  \draw[fill=gray!40!white,gray!40!white] plot[smooth,samples=100,domain=-2.243:2.243] (\x,\x*\x+0.25) --  (0,5.34);


  \draw[smooth,samples=100,domain=0.3/5.43:4.9, thick, color = gray!80!white]
  plot(\x,0.3/\x);
  \draw[smooth,samples=100,domain=1.5/5.43:4.9, thick, color = gray!80!white]
  plot(\x,1.5/\x);
  \draw[smooth,samples=100,domain=3/5.43:4.9, thick, color = gray!80!white]
  plot(\x,3/\x);
   \draw[smooth,samples=100,domain=5/5.43:4.9, thick, color = gray!80!white]
  plot(\x,5/\x);
  \draw[smooth,samples=100,domain=8/5.43:4.9, thick, color = gray!80!white]
  plot(\x,8/\x);
  \draw[smooth,samples=100,domain=12/5.43:4.9, thick, color = gray!80!white]
  plot(\x,12/\x);
    \draw[->] (-0.3,0.25) -- (0.8,0.25); 
  \node[right] at (0.8,0.25){$x_{m+1} = \frac{1}{4}$};
  \draw[-latex,color=black] (-5,0) -- (5,0);   
  \draw[-latex,color=black] (0,-1) -- (0,5.75); 

 \node[right] at (5,0){$x$};
 \node[above] at (0,5.75){$x_{m+1}$};
 \draw[smooth,samples=100,domain=0:2.3, ultra thick, color = black]
  plot(0,\x*\x+0.25);

  \draw[smooth,samples=100,domain=-0:4.9, thick, color = black, dashed]
  plot(\x,\x);
   \draw[smooth,samples=100,domain=-0:-4.9,  thick, color = black, dashed]
  plot(\x,-\x);
\end{tikzpicture}
\end{center}

\caption{ The set $S = \{(x,x_{m+1}) \in \R^{2}: q(x,x_{m+1}):= x_{m+1} \ge \frac{1}{4} - x^2 \ge 0\}$ is illustrated in light grey, while the bold line
illustrates the set
$S \cap q^{-1}(0)$. The dotted line illustrates the second-order cone.  The
grey curves illustrate the level sets of $p(x,x_{m+1}) = x_{m+1}x$.}
\label{fig:lasserre}
\label{fig:sturm}
\end{figure}
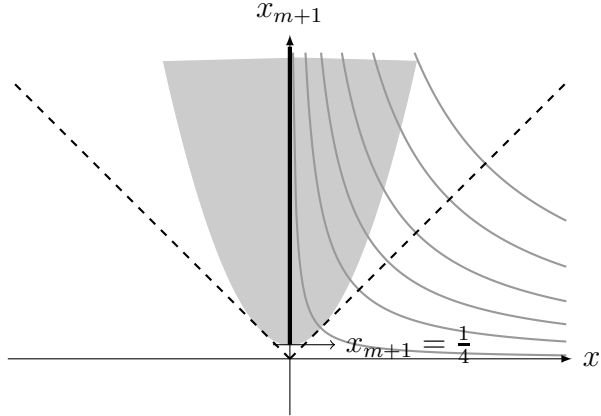

\section{Inequality extension}
\label{sec:inequalities}
In this section,
we provide an extension of Theorem~\ref{main.thm.equ}
to cases in which the polynomial $h$ defines an inequality instead of an equality.
That is, the set over which the non-negativity of a polynomial is certified is of
the form $S \cap h^{-1}(\R_+)$, where $S \subseteq \R^n$, $h \in \R_d[x]$, and
\[
h^{-1}(\R_+):= \{ x \in \R^n: h(x) \ge 0\}.
\]

\begin{corollary}\label{col:finerIneq}
Assume $K\subseteq \R^n$ is a pointed closed convex cone. Let $S\subseteq K$ be a closed set and
let~$h \in \R[x]$ with $\deg(h)\le d/2$.
If $p(x) \in \P_{d}(S\cap h^{-1}(\R_+))$ and
$\tilde p(x) \in \P_{d}( \rec{S} \cap \tilde h^{-1}(\R_+))$, then
$p(x) \in \closure(\{F(x,h(x)): F(x,u) \in \P_{d}(S\times \R_+))$.
\end{corollary}

\begin{proof} Let $p \in \P_d(S\cap h^{-1}(\R_+))$ such that $\tilde p(x) \in \P_{d}( \rec{S} \cap \tilde h^{-1}(\R_+))$
be given. Now let $T = S\times \R_+$ and $g(x,t) = (1+a\tr x)^{d-2\deg(h)}(h(x)-t)^2$ where $a$ is defined by \eqref{eq:adef}.
Then $p(x) \in \P_d(T \cap g^{-1}(0))$. We claim that $\rec{T} \cap \tilde{g}^{-1}(0) \subseteq
(\rec{S} \cap \tilde h^{-1}(\R_+))\times \R_+$ and thus, by hypothesis, $\tilde p(x) \in \P_d(\rec{T} \cap \tilde{g}^{-1}(0))$. Therefore, from Theorem~\ref{thm:finerEq}, $p(x) = \lim_{i \rightarrow \infty} r_i(x,t) + g(x,t) q_i$ with $r_i(x,t) \in \P_d(T)$ and $q_i \in \R$. Taking $t = h(x)$ we obtain $p(x) = \lim_{i \rightarrow \infty} r_i(x,h(x)) \in \closure(\{F(x,h(x)): F(x,u) \in \P_d(S\times \R_+)\})$.
To finish the proof, we prove the claim.
First, assume that $\deg(h) > 1$.  In this case $\rec{T} \cap \tilde{g}^{-1}(0) = \rec{S}\times \R_+ \cap \{(x,t):\tilde h(x) = 0\}  = (\rec{S} \cap \tilde h^{-1}(0))\times \R_+ \subseteq
(\rec{S} \cap \tilde h^{-1}(\R_+))\times \R_+$. Second, assume that $\deg(h) = 1$.  In this case $\rec{T} \cap \tilde{g}^{-1}(0) = \rec{S}\times \R_+ \cap (\{(x,t):\tilde h(x) - t = 0\} \cup \{(0,t): t \in \R\})  \subseteq
(\rec{S} \cap \tilde h^{-1}(\R_+))\times \R_+$. \end{proof}

\begin{theorem}
\label{thm:ineq}
Assume $K\subseteq \R^n$ be a pointed closed convex cone. Let $S\subseteq K$ be a closed set and  $h \in \R[x]$ be such that $\deg h \le d/2$. If $\rec{(S \cap h^{-1}(\R_+))} = \rec{S} \cap \tilde h^{-1}(\R_+)$, then
$
 \P_d(S\cap h^{-1}(\R_+)) =  \closure(\{F(x,h(x)): F(x,u) \in \P_{d}(S\times \R_+)\}) \cap \R_d[x],
$
\end{theorem}

\begin{proof}
Clearly, $\closure(\{F(x,h(x)): F(x,u) \in \P_{d}(S\times \R_+)\}) \cap \R_d[x] \subseteq  \P_d(S\cap h^{-1}(\R_+))$. For the other direction, notice that if $p(x) \in  \P_d(S\cap h^{-1}(\R_+))$ then
by Proposition~\ref{prop:inf},
$\tilde p(x) \in \P_d(\rec{(S\cap h^{-1}(\R_+))}) = \P_d(\rec{S}\cap \tilde h^{-1}(\R_+))$. Thus, Corollary~\ref{col:finerIneq} implies $p(x) \in \closure(\{F(x,h(x)): F(x,u) \in \P_{d}(S\times \R_+)\})$.
\end{proof}

\subsection{Quadratic Modules without the Strong Moment Property}
\label{sec:Netzer}

We now show
how
Theorem~\ref{thm:ineq} can be used to define
a hierarchy of LMI approximations that is guaranteed to converge
to the non-negative polynomials in a given set, when such type of hierarchy cannot be obtained using a quadratic module construction. Formally,
the quadratic module $\QM(g_1,\dots,g_m)$ associated to a semialgebraic set
$U = \{x \in \R^n: g_j(x) \ge 0, j=1,\dots,m\}$ is said to satisfy the {\em strong moment property} (SMP) if
%
$\closure(\QM(g_1,\dots,g_m)) = \P(U)$ \citep[cf.,][Prop. 2.6, cond. (1)]{cimprivc2011}.
\citet{Netz09} characterizes whether
$\QM(g_1,\dots,g_m)$ satisfies the SMP
using the notion of {\em tentacles}.

\begin{definition}[Tentacles]
Given a compact set $K \subseteq \R^n$ with nonempty interior, a {\em tentacle} of $K$ in direction $z$ is the set
$T_{K,z} := \left \{(\lambda^{z_1}x_1,\ldots, \lambda^{z_n}x_n) : \lambda \ge 1, x = (x_1,\dots,x_n) \in K \right \}$.
\end{definition}

Below, we consider an example of a quadratic module similar
to the one presented in~\cite[][Section 6]{Netz09}, that does not satisfy the SMP.

\begin{example}
\label{ex:netzer}
Consider the set
$U = \left \{ (x_1, x_2) \in \R^2_+: (x_2-x_1^2)(2x_1^2-x_2) \ge 0 \right \}$
illustrated in Figure~\ref{fig:netzer}. Note that $T_{K,z} \subseteq U$ by letting $K= \{(x_1,x_2) \in \R^2: (x_1 -1)^2 + (x_2 - \frac{3}{2}) \le (\frac{1}{5})^2\}$ and $z = (1,2) \in \R^2$.
From \cite[][Theorem~5.2]{Netz09} it follows that the quadratic module $\QM(x_1, x_2, (x_2-x_1^2)(2x_1^2-x_2))$ does not satisfy the SMP \cite[cf.,][Section 6]{Netz09}. Thus $\closure(\QM(x_1, x_2, (x_2-x_1^2)(2x_1^2-x_2))) \subset \P(U)$.
\end{example} 

%

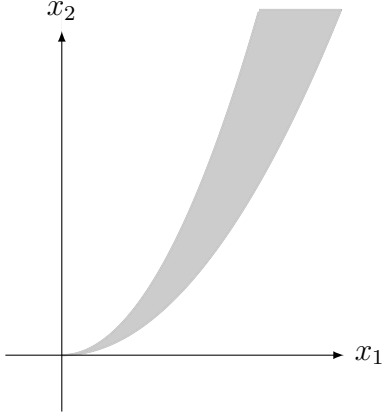
\begin{figure}
\begin{center}
\begin{tikzpicture}[scale = 0.75]
 \draw[fill=gray!40!white] plot[smooth,samples=100,domain=0:2.475] (2*\x,\x*\x) --  (0,6.125);
  \draw[fill=white] plot[smooth,samples=100,domain=0:1.75] (2*\x,2*\x*\x) --  (0,6.125);
   \draw[smooth,samples=100,domain=0:2.475, thick, color = gray!40!white]
  plot(2*\x, \x*\x);
   \draw[smooth,samples=100,domain=0:1.75,  thick, color = gray!40!white]
  plot(2*\x, 2*\x*\x);
\draw[color=gray!40!white, thick] (2*1.75,6.125) -- (2*2.475,6.125);
\draw[color=white, very thick] (0,6.125) -- (2*1.75,6.125);
  \draw[-latex,color=black] (-1,0) -- (5,0);
  \draw[-latex,color=black] (0,-1) -- (0,5.75);
 \node[right] at (5,0){$x_1$};
 \node[above] at (0,5.75){$x_2$};
\end{tikzpicture}
\end{center}
\caption{Illustration of the set
$U =\left \{ x \in \R^2_+: (x_2-x_1^2)(2x_1^2-x_2) \ge 0 \right \}$ (in gray).}
\label{fig:netzer}
\end{figure}

A consequence of a quadratic module $\QM(g_1,\dots,g_m)$
not satisfying the SMP is that the non-negativity of a polynomial on the
associated set $U = \{x \in \R^n: g_j(x) \ge 0, j=1,\dots,m\}$
is not guaranteed to be certifiable using a LMI formulation
based on the {\em truncated} quadratic module~\cite[cf.,][]{BlekPT13}
that underlies the use of Putinar's Positivstellensatz to obtain
LMI approximations of PO problems with a compact feasible
set~\cite[see, e.g., the seminal work in][]{Lass02b, Lass01}.
As illustrated in the next proposition, in such cases, the certificate
of non-negativity in Theorem~\ref{thm:ineq} can be used to obtain a hierarchy
of LMI problems that is (in the limit) guaranteed to certify the non-negativity
of a polynomial on the set of interest, provided that the set satisfies
the appropriate horizon cone conditions.

\begin{proposition}
\label{prop:apxnetzer}
Let $U = \left \{ (x_1, x_2) \in \R^2_+: (x_2-x_1^2)(2x_1^2-x_2) \ge 0 \right \}$, and
$d \ge 8$, then:
\begin{enumerate}[(i)]
\item
\label{it:onenetzer}
$\P_d(U) =  \closure(\{F(x_1,x_2,(x_1-x_2^2)(2x_1^2-x_2)): F(x_1,x_2,u) \in \P_{d}(\R^3_+)\}) \cap \R_d[x]$. %
\item
\label{it:twonetzer}
There exists a LMI hierarchy $Q^r, r=1,\dots,$ such that
$\P_d(U) = \bigcup_{r=1}^{\infty} Q^r\cap \R_d[x]$.

\end{enumerate}
\end{proposition}

\begin{proof}
To show~\eqref{it:onenetzer},
let $S =  \R^2_+$, and $h(x_1,x_2) = (x_2-x_1^2)(2x_1^2-x_2)$, so
that $U = S \cap h^{-1}(\R_+)$. Clearly, $S$ is a pointed closed convex cone. Also,
the horizon cone condition $\rec{S} \cap \tilde h^{-1}(\R_+) = \rec{(S \cap h^{-1}(\R_+))}$
is satisfied (cf., Proposition~\ref{prop:infcond3}). The result then follows from Theorem~\ref{thm:ineq}.
To show~\eqref{it:twonetzer}, let $n=2$ and define the LMI hierarchy
\begin{equation}
\label{eq:PolyaApx}
Q^r = \left \{ p \in \R_{d}[x]:    (1 + e\tr x  + h(x) )^r p(x) = \sum_{
(\alpha,\alpha') \in \Gamma^{d+r\deg(h)}}
c_{\alpha, \alpha'} x^{\alpha} h(x)^{\alpha'}, c_{\alpha, \alpha'}  \ge 0\right \},
\end{equation}
where $e \in \R^n$ is the vector of all ones, $x^{\alpha} := x_1^{\alpha_1} \cdots
x_n^{\alpha_n}$, and $\Gamma^r = \{ (\alpha,\alpha')  \in \N^{n+1}: \|(\alpha,\alpha')\|_1 \le r\}$. From Polya's
Positivstellensatz \cite[see, e.g.,][]{HardLP88} it follows that $\bigcup_{r=1}^{\infty} Q^r = \closure(\{F(x,h(x)): F(x,u) \in \P_{d}(\R^{n+1}_+)\})$. Thus, we obtain the desired LMI hierarchy.
\end{proof}

Notice that in Proposition~\ref{prop:apxnetzer},
a polyhedral approximation of $\P(\R^n_+)$; that is, the
{\em copositive} polynomials \cite[cf.][]{Bomz09, Bure10, Dur10} based on
Polya's Theorem is used in conjunction with Theorem~\ref{thm:ineq}
to obtain the desired approximation for $\P_d(U)$. Given the rich
knowledge of LMI approximations to the set of copositive polynomials
\cite[see, e.g.,][]{BomzSU12, BomzO14}, the hierarchy defined in \eqref{eq:PolyaApx}
is one of the many potential ones that could be used.

\subsection{Polyhedral hierarchies}
\label{sec:balls}

As illustrated by  \eqref{eq:PolyaApx}, the certificates of non-negativity
introduced here readily allow for the use of rational polynomial expressions
to certify the non-negativity of a polynomial. The existence of such rational
certificates is in general guaranteed by the Krivine-Stengle Positivstellensatz \citep{Kriv64,Sten74}. However, as discussed in detail in \citet{JeyaLL13}
finding such certificates cannot be readily formulated as a LMI. Instead, the certificate of non-negativity in Proposition~\ref{prop:apxnetzer}(\ref{it:onenetzer}) can
lead to LMI
certificates of non-negativity for a polynomial. In~Proposition~\ref{prop:apxnetzer}(\ref{it:twonetzer}),
this is
thanks to the denominator being a fixed given polynomial (i.e., $(1 + e\tr x  + h(x) )$
in~\eqref{eq:PolyaApx}). A noteworthy characteristic of such certificates is
that they can lead to {\em low rank}
 certificates even when only polyhedral
LMIs are considered to certify the non-negativity of the polynomial of interest.
By {\em low rank} we mean here that the polynomials involved
in the certificate are of low degree.  This is in contrast to the limited research
on using polyhedral LMIs to certificate the non-negativity of a polynomial
\cite[see, e.g.,][for noteworthy exceptions]{Lass02e, deKlP02, KarlMN11, SagoY13}. To formally illustrate
this, first consider the following result.

\begin{corollary}
\label{cor:compact}
Let $S = \{x \in \R^n_+: g_j(x) \ge 0, j=1,\dots,m\}$ be a compact set, and $p \in \R[x]$ satisfy $p(x) > 0$ for all $x \in S$. Then,
there exists $r \ge 0$, such that:
\begin{equation}
\label{eq:copocert}
\left (1 + \sum_{i=1}^n x_i + \sum_{j=1}^m g_j(x) \right )^r p(x) =
 \sum_{(\alpha, \beta) \in \mathbb{N}^{n+m}} c_{\alpha, \beta} x^{\alpha} g(x)^{\beta},
\end{equation}
for some $c_{\alpha, \beta} \ge 0$ for all $(\alpha, \beta) \in \mathbb{N}^{n+m}$. Where we use
$x^{\alpha} := \Pi_{i=1}^n x_i^{\alpha_i}$, and $g(x)^{\beta} := \Pi_{j=1}^m g_j(x)^{\beta_j}$.
\end{corollary}

\begin{proof}
It follows from Theorem~\ref{thm:ineq} and Polya's Theorem~\cite[see, e.g.,][]{HardLP88}.
\end{proof}

Now we show that the polyhedral hierarchy \eqref{eq:copocert}  can have low rank convergence by considering an instance of the
 {\em Celis-Dennis-Tapia} (CDT) problem~\cite[cf.,][]{CeliDT85}.
This classical problem is concerned with the non-negativity
of a quadratic polynomial over the intersection of two ellipses.
Recent advances on this problem have been made thanks to the use of
polynomial optimization techniques~\cite[cf.,][]{BomzO14}.
Specifically, for $n\ge 3$, consider the polynomial $q \in \R[x]$:
\[
q(x) := -2x_1 + 8x_1\sum_{i=1}^n x_i.
\]
Note that $q$ is not a {\em copositive} polynomial \cite[cf.,][]{Bomz09, Bure10, Dur10}; that is, $q \not \in \P(\R^n_+)$. In particular,
\[
q(x_1,0,\dots,0) < 0 \text{ for } 0 < x < 1/4.
\]
However, we can use Corollary~\ref{cor:compact} to certify that $q \in \P(\B_e \cap \B_{\sfrac e2})$, where
\[
\B_c = \left \{x \in \R^n_+: b_c(x):= 1-\|x-c\|^2 \ge 0 \right \},
\]
is the unitary ball centered at $c \in \R^n$.
In particular notice that:
\begin{multline}
\left ( 1 + \sum_{i=1}^n x_i + b_e(x) + b_{\sfrac e2}(x) \right )q(x) = \\ \left (8x_1\sum_{i=1}^nx_i \right )\left (b_e(x) + b_{\sfrac e2}(x) + \sum_{i=1}^nx_i \right ) +
2x_1\left (\left (\frac{5}{4}n-3 \right ) + 2\sum_{i=2}^n x_i^2 \right ),
\end{multline}
which after expanding the right hand side, has the form of \refer{eq:copocert} with $r=1$, and, in particular, certifies that $q$ is non-negative
in the intersection of the unitary balls centered at $e$, and $\frac{1}{2}e$.

\section{Further extensions}
\label{sec:extensions}
In this section we show that the assumptions of Theorem~\ref{main.thm.equ}
regarding the non-negativity of $h$ on $S$ and the existence of a pointed
cone $K$ such that $S \subseteq K$ can be relaxed by doubling the
degree of the polynomials involved in the certificate of non-negativity.
First, let us consider the non-negativity assumption on $h$ in Theorem~\ref{main.thm.equ}.

\begin{corollary}~\label{the.corol}
The statement of Theorem~\ref{main.thm.equ} holds if the hypothesis $h \in \P_d(S)$ is changed to $\deg(h) \le d/2$.
\end{corollary}
\begin{proof}
Let $h_1(x) = h(x)^2 \in \P_d(S)$. We have
\begin{align*}
\rec{(S \cap h_1^{-1}(0))}
&= \rec{(S \cap h^{-1}(0))}\\
&= \rec{S} \cap \tilde h^{-1}(0) & \text{(by hypothesis)}\\
&= \rec{S} \cap \tilde h_1^{-1}(0) &\text{(as $\tilde{h}_1 = \tilde{h}^2$)}.
\end{align*}
Applying Theorem \ref{main.thm.equ}, we have that $\rec{(S \cap h_1^{-1}(0))}  = \rec{S} \cap \tilde h_1^{-1}(0)$ implies:
\begin{align*}
 \P_d(S\cap h^{-1}(0)) &= \P_d(S\cap h_1^{-1}(0)) \\& = \closure(\P_d(S) + h_1(x)\R_{d-\deg(h_1)}[x]) \\ &\subseteq \closure(\P_d(S) + h(x)\R_{d-\deg(h)}[x])
 \\ &\subseteq  \P_d(S\cap h^{-1}(0)).
 \end{align*}
 For the implication in the other direction, note that if there is $s \in \rec{S} \cap \tilde{h}^{-1}(0)$ such that $s \not \in \rec{(S \cap h^{-1}(0))}$, then $ s\in \rec{S} \cap \tilde{h}_1^{-1}(0)$, and
 $s \not \in \rec{(S \cap h_1^{-1}(0))}$, and the counter-example in Section~\ref{sec:counter} follows with
 $h \to h_1$.
\end{proof}

Now we consider the extension of Theorem~\ref{main.thm.equ} to the case in which the underlying set $S$ is not contained on a pointed cone.
For that purpose, we first state the following lemma.

\begin{lemma}\label{lem:red2+}
Let $U \subseteq \R^n$ be a closed set.  Then $\rec{\{(x,y) \in \R^{2n}_+: x-y \in U\}} = \{(x,y) \in \R^{2n}_+: x-y \in \rec U\}$
\end{lemma}
\begin{proof}
To show $\rec{\{(x,y) \in \R^{2n}_+: x-y \in U\}} \subseteq \{(x,y) \in \R^{2n}_+: x-y \in \rec U\}$,
 let $(x,y) \in \rec{\{(x,y) \in \R^{2n}_+: x-y \in U\}}$. Then, there are $(x_k,y_k)$ and $\lambda_k\downarrow 0$ such that $x_k-y_k \in U$ and $(x,y) = \lim_{k \rightarrow \infty} \lambda_k(x_k,y_k)$. Then, $x-y = \lim_{k \rightarrow \infty} \lambda_k (x_k - y_k)\in \rec U$. For the other direction, let $(x,y) \in \R^{2n}_+$ be such that $x-y \in \rec U$, there are $z_k \in U$ and $\lambda_k\downarrow 0$ such that $x-y = \lim_{k \rightarrow \infty} \lambda_k z_k$.
For any $s \in \R^n$, let $s^+ = \max(s,0)$ and $s^- = \max(-s,0)$ where the max is applied coordinate-wise. Without loss of generality, assume $\lambda_k \neq 0$. Define $x_k = (x-(x-y)^+)/\lambda_k + z_k^+$ and $y_k = (y-(x-y)^-)/\lambda_k + z_k^-$. Then for all $k=1,\dots$, we have $x_k,y_k \geq 0$ and $x_k - y_k = (x-y-((x-y)^+ - (x-y)^-))/\lambda_k + z_k^+ - z_k^- = z_k \in U$. Notice that $\lim_{k \rightarrow \infty}\lambda_k z_k^+ = (x-y)^+$ and $\lim_{k \rightarrow \infty}\lambda_kz_k^- = (x-y)^-$ and thus  $ \lim_{k \rightarrow \infty} (x_k , y_k)\lambda_k = (x,y)$.
\end{proof}

\begin{theorem}
\label{th:nonconic}
Let $S\subseteq \R^n$ be a closed set and $h\in \P_d(S)$ be such that
 $\rec{(S \cap h^{-1}(0))} = \rec{S} \cap \tilde h^{-1}(0).$
Then
\begin{equation*}
 \P_d(S\cap h^{-1}(0)) = \closure(\P_{2d}(S) + h(x)\R_{2(d-\deg(h))}[x]) \cap \R_d[x].
\end{equation*}
\end{theorem}

\begin{proof}
From Lemma~\ref{lem:subset},  $\closure(\P_{2d}(S) + h(x)\R_{2(d-\deg(h))}[x]) \cap \R_d[x] \subseteq \P_d(S\cap h^{-1}(0))$. Now we show,
\[
\P_d(S\cap h^{-1}(0)) \subseteq \closure(\P_{2d}(S) + h(x)\R_{2(d-\deg(h))}[x]).
\]
Define $T = \{(z,y)\in \R^{2n}_+ : z-y \in S\}$ and $g(z,y) = h(z-y)$. Then,
\[
\begin{array}{lllll}
\rec{(T \cap g^{-1}(0))}
&= \rec{\{(z,y)\in \R^{2n}_+ : z-y \in S \cap h^{-1}(0) \}}\\
& = \{(z,y)\in \R^{2n}_+ : z-y \in \rec{ (S \cap h^{-1}(0))} \} & \text{(by Lemma \ref{lem:red2+})}\\
& = \{(z,y)\in \R^{2n}_+ : z-y \in \rec{ S} \cap \tilde h^{-1}(0) \} & \text{(by hypothesis)}\\
& = \{(z,y)\in \R^{2n}_+ : z-y \in \rec{ S} \} \cap \tilde g^{-1}(0) & \text{ }\\
&=  \rec{T} \cap \tilde g^{-1}(0)& \text{(by Lemma \ref{lem:red2+}).}
\end{array}
\]
From Theorem \ref{main.thm.equ},
\begin{equation*}
\P_d(T\cap g^{-1}(0)) = \closure(\P_d(T) + g(z,y)\R_{d-\deg(g)}[z,y]).
\end{equation*}
Now, let $p(x) \in \P_d(S\cap h^{-1}(0))$, then $\hat{p}(z,y) = p(z-y) \in \P_d(T\cap g^{-1}(0))$. Thus, there are $r_i(z,y) \in \P_d(T)$ and $q_i(z,y) \in \R_{d-\deg(h)}[z,y]$ for $i=1,2,\dots$ such that $\hat{p}(z,y) = \lim_{i \rightarrow \infty} r_i(z,y) + g(z,y)q_i(z,y)$. Therefore
\begin{align*}
p(x)
= \hat{p}(x^2+x+1,x^2+1)
=\lim_{i \rightarrow \infty} r_i(x^2+x+1,x^2+1) + h(x)q_i(x^2+x+1,x^2+1),
\end{align*}
but $r_i(x^2+x+1,x^2+1) \in \P_{2d}(S)$ and $\deg(q_i(x^2+x+1,x^2+1)) \leq 2(d-\deg(h))$ for every~$i$.
\end{proof}

\section{Proof of Theorem~\ref{thm:finerEqGen}}
\label{sec:proofs}
Following the statement of Theorem~\ref{thm:finerEqGen}, throughout this Section, let $K \subseteq \R^n$ be a given  closed convex pointed  cone.

The proof of Theorem~\ref{thm:finerEqGen} relies on Proposition~\ref{th:compact} via a suitable {\em compactification} procedure of the set $S\cap h^{-1}(0)$ (cf., \eqref{eq:finer}). Specifically, fix a point
\begin{equation}
\label{eq:adef}
a\in \R^n \text{ such that } a\tr x > 0 \text{ for all } x \in K\setminus\{0\},
\end{equation}
and for any $x\in \R^n$  let
\begin{equation}
\label{eq:comp}
\ns{x}:= \frac{1}{1+a\tr x} \matr{1 \\ x}.
\end{equation}
Also, for $p \in \R_d[x]$, let $\ns{p} \in \R[x_0,x]$ be defined by
\begin{equation}
\label{eq:popbar}
\ns{p}(x_0,x_1,\dots,x_n) := p(x_1/x_0,\dots,x_n/x_0)\cdot x_0^{d}.
\end{equation}
Notice that \eqref{eq:popbar} corresponds to the natural homogenization of a polynomial of degree $d$.
Furthermore, if $\deg{p} <d$, then $\ns{p}$ is a homogeneous polynomial of degree $d$.
Finally, for~$S \subseteq K$ let
\begin{equation}
\label{eq:setbar}
\ns{S}:=\text{closure}\{\ns{x}:x \in S\}.
\end{equation}
\begin{lemma}
\label{lem:alpha}
Let $a \in \R^n$ be defined by~\eqref{eq:adef}. Then, there
exists $\alpha > 0$ such that $a\tr x \ge \alpha \|x\|$ for all $x \in K$.
\end{lemma}
\begin{proof}
Let $B = \{x\in K: \|x\| = 1\}$, $a \in \R^n$ be defined by~\eqref{eq:adef}, and $\alpha := \inf \{a\tr x: x \in B\} > 0$. As $B \subseteq \R^n$ is compact, $\alpha$ is attained. Thus, for any $x \in K\setminus \{0\}$, $a\tr x =  \left(a\tr \frac{x}{\|x\|} \right)  \|x\|\ge \alpha \|x\|$.
\end{proof}
\begin{lemma}\label{lem:barCompact} For any $S \subseteq K$, the set $\ns{S}$ is compact.
\end{lemma}
\begin{proof} By construction, $\ns{S}$ is closed, thus it is enough to show $\ns{S}$ is bounded.
From Lemma~\ref{lem:alpha}, it follows that $a\tr x \ge \alpha \|x\|$ for all $x \in K$.
 Thus, for any $x \in S \setminus \{0\}$ we have $\|\ns{x}\| \leq \frac { 1 + \|x\|}{1 + \alpha\|x\|} \leq \frac 1{\alpha}$.
\end{proof}
The map $x \in \R^n \mapsto \ns x \in \R^{n+1}$ in \eqref{eq:comp} establishes a natural correspondence between the cone~$K$ and a bounded slice of the cone $\R_+\times K$.
Loosely speaking, $\rec{S}$ is the boundary ``at infinitum'' of $S$.
Similar to the natural correspondence between the sets
$\{\ns{x}:x \in S\}$ and $S$,
the set $\ns{S}$ corresponds in a natural way to the set $S \cup \rec{S}$.
This correspondence also extends naturally to polynomials over~$S$ in such way that non-negativity of the
polynomial is preserved.
 To formally state this correspondence, we first state some basic properties of the
 maps~\eqref{eq:popbar} and~\eqref{eq:setbar}.
 For the purpose of brevity, in what follows, we use $(x_0, x) \in A$ for any set $A \subseteq \R^{n+1}$
to indicate $(x_0, x\tr)\tr \in A$,
where $x_0 \in \R$, $x \in \R^n$.

\begin{lemma}\label{lem:tilde+}Let $p \in \R_d[x]$  and $S \subseteq K$. Then
\begin{enumerate}[(i)]
\item \label{it:2} For every $x \in \R^{n}$, $\tilde p(x) = \ns p(0,x)$ for any $p$ of degree $d$.
\item \label{it:3} $p(x) = \ns{p}(\overline x)(1+a\tr x)^{d}$.
\item \label{it:4} $p \in \P_d(S)$ if and only if $\ns{p} \in \P_d(\ns{S})$.
\end{enumerate}
\end{lemma}
\begin{proof}
Properties~\eqref{it:2} and \eqref{it:3} readily follow from the definitions of $\tilde p$ (the homogeneous component of $p$ of highest degree), $\ns{p}$ (in \eqref{eq:popbar}) and $\ns x$ (in \eqref{eq:comp}). Property~\eqref{it:4} directly follows from~\eqref{it:3}.
\end{proof}

With Lemma~\ref{lem:tilde+}, next we characterize the $\closure(\P_d(S) + h(x)\R)$ (cf., \eqref{eq:finer}) using the
maps~\eqref{eq:popbar} and~\eqref{eq:setbar}.

\begin{lemma}\label{lem:barclosure}
Let $S\subseteq K$ and $h \in \R_d[x]$ be of degree $d$.
Then
$
\closure(\P_d(S) + h(x)\R)
=  \{p(x)\in \R_d[x]: \ns p(x_0,x) \in \P_d( \ns{S} \cap \ns h^{-1}(0))\}$.
\end{lemma}

\begin{proof}
First, we show that $\closure(\P_d(S) + h(x)\R)
\subseteq \{p(x)\in \R_d[x]: \ns p(x_0,x) \in \P_d( \ns{S} \cap \ns h^{-1}(0))\}$.
Since $\{p(x)\in \R_d[x]: \ns p(x_0,x) \in \P_d( \ns{S} \cap \ns h^{-1}(0))\}$ is closed, it is enough to show
that $\P_d(S) + h(x)\R
\subseteq  \{p(x)\in \R_d[x]: \ns p(x_0,x) \in \P_d( \ns{S} \cap \ns h^{-1}(0))\}$. Let $p(x) = f(x) + qh(x)$
with $f(x) \in \P_d(S)$ and $q \in \R$. Also, let $y:= (x_0,x) \in \ns{S} \cap \ns h^{-1}(0)$, thus $y = \lim_{k\to \infty} \ns x_k$ with $x_k \in S$. Note that
 \begin{align*}
 \ns p(y) &= \lim_{k\to \infty}\ns p(\ns x_k)\\
 & = \lim_{k\to \infty} p(x_k)(1+a\tr x_k)^{-d}  &\text{(by Lemma \ref{lem:tilde+}\eqref{it:3})}\\
  &= \lim_{k\to \infty}(f(x_k) + q h(x_k))(1+a\tr x_k)^{-d}\\
   &= \lim_{k\to \infty} \ns f(\ns x_k) + q\ns h(\ns x_k)&\text{(by Lemma \ref{lem:tilde+}\eqref{it:3})}\\
 &= \ns f(y) + q\ns h(y)\\
 & \ge 0 &\text{(by Lemma \ref{lem:tilde+}\eqref{it:4} and using $\ns h(y) = 0$)}.
 \end{align*}
Now, we show that $\closure(\P_d(S) + h(x)\R)
\supseteq \{p(x)\in \R_d[x]: \ns p(x_0,x) \in \P_d( \ns{S} \cap \ns h^{-1}(0))\}$.
Let $p(x)\in \R_d[x]$ be such that  $\ns p(x_0,x) \in \int\P_d( \ns{S} \cap \ns h^{-1}(0))$, which is non-empty by Lemma~\ref{lem:jaja}~\eqref{lem:jajaint1}.
Since $\ns S$ is compact, then from Proposition~\ref{th:compact} it follows
that $\ns{p}(x_0,x) \in \P_d(\ns S) + \ns h(x)\R$.
Therefore, there exist $f(x_0,x) \in \P_d(\ns S)$ and $q \in \R$ such that $\ns p(x_0,x) = f(x_0,x) + q \ns h(x_0,x)$.
Thus
\begin{align*}
 p(x) &= \ns p (\ns x)(1+a\tr x)^d & \text{(by Lemma \ref{lem:tilde+}\eqref{it:3})}\\
  &=  (f(\ns x) + q\ns h(\ns x))(1+a\tr x)^d\\
   &= f(\ns x)(1+a\tr x)^d + q h(x) &  \text{(by Lemma \ref{lem:tilde+}\eqref{it:3})}.
   \end{align*}
For any polynomial $f \in \R_d[x_0,x]$, $f(\ns x)(1+a\tr x)^d$ is a polynomial, thus we have $p(x) \in \P_d(S) + h(x)\R$, that is
\[
 \{p(x) \in \R_d[x]: \ns p(x_0,x) \in \int\P_d( \ns{S} \cap \ns h^{-1}(0))\}
\subseteq  \P_d(S) + h(x)\R.
\]
Therefore
\begin{align*}
 \{p(x) \in \R_d[x]: \ns p(x_0,x) \in \P_d( \ns{S} \cap \ns h^{-1}(0))\}
&\subseteq \closure\{p(x) \in \R_d[x]: \ns p(x_0,x) \in \int\P_d( \ns{S} \cap \ns h^{-1}(0))\}\\
&\subseteq  \closure(\P_d(S) + h(x)\R).
\end{align*}
\end{proof}

\begin{lemma}
\label{lem:condition}
Let $S \subseteq K$. Then
$\{x \in \rec{S}: a\tr x = 1\} =  \{ x \in K: (0,x) \in \ns{S}\}$.
\end{lemma}

\begin{proof} To show
 $\{x \in \rec{S}: a\tr x = 1\}  \supseteq \{ x \in K: (0,x) \in \ns{S}\}$, let $(0, x) \in \ns{S}$ with $x \in K$. Then, there are $x_k \in S$, $k = 1,\dots$ such that
\[
\begin{pmatrix}
0 \\ x
\end{pmatrix} = \lim_{k \to \infty,} \begin{pmatrix}  \frac{1}{1+a\tr x_k} \\  \frac{x_k}{1+a\tr x_k}\end{pmatrix}.
\]
Define $\lambda_k  = \frac{1}{1+a\tr x_k}$, using \eqref{eq:adef} we have $\lambda_k > 0$
for all $k=1,\dots$. Also
$\lim_{k \to \infty} \lambda_k  = 0$, and $x = \lim_{k \to \infty} \lambda_k x_k$. Taking a subsequence if necessary, we can assume $\lambda_k \downarrow 0$ and
it follows that~$x \in \rec{S}$.
Furthermore, notice that
\[
a\tr x = \lim_{k \to \infty} \frac{a\tr x_k}{1 + a\tr x_k} =  1-\lim_{k \to \infty}\frac{1}{1 + a\tr x_k}= 1.
\]
To show
 $\{x \in \rec{S}: a\tr x = 1\}  \subseteq \{ x \in K: (0,x) \in \ns{S}\}$,
 let $x \in \rec{S}$ be such that $a\tr x = 1$. Then $x = \lim_{k \to \infty} \lambda_k x_k$, where  $\lambda_k \downarrow 0$, and $x_k \in S, k=1,\dots$.  Also, \[\lim_{k \to \infty} \lambda_k a\tr x_k = a\tr x = 1.\]
Thus we can write
\begin{align*}
x
 =  \frac{x}{0 + 1}
 =  \frac{\lim_{k \to \infty} \lambda_k x_k}{\lim_{k \to \infty}  \lambda_k + \lambda_k (a\tr x_k)}
 = \lim_{k \to \infty} \frac{\lambda_k x_k}{\lambda_k (1+ a\tr x_k)}
 = \lim_{k \to \infty} \frac{x_k}{1 + a\tr x_k},
\end{align*}
and
\begin{align*}
0 = \frac 0{0+1}
 =  \frac{\lim_{k \to \infty} \lambda_k }{\lim_{k \to \infty} \lambda_k +\lim_{k \to \infty} \lambda_k (a\tr x_k)}
 = \lim_{k \to \infty} \frac{\lambda_k }{\lambda_k (1+ a\tr x_k)}
 = \lim_{k \to \infty} \frac{1}{1 + a\tr x_k}.
\end{align*}
Thus $(0, x) \in \ns{S}$.
\end{proof}

\begin{lemma}
\label{lem:z0pos}
Let $S \subseteq K$ and
$(x_0,x) \in \ns{S}$ be such that $x_0 > 0$.  Then $(x_0,x) = \ns{x}$ where $x \in \closure\{S\}$.
\end{lemma}

\begin{proof}
Let $y:=(x_0, x) \in \ns{S}$ with $x_0 > 0$ be given. Also, let $x_k \in S, k=1,\dots$ be a sequence
satisfying $ y = \lim_{k \to \infty} ( \frac{1}{1+ a\tr x_k}, \frac{x_k\tr}{1+ a\tr x_k})\tr$.
From $x_0 = \lim_{k \to \infty} \frac{1}{1+ a\tr x_k} >0$
it follows that the sequence $x_k \in S, k=1,\dots$ is bounded. Therefore, there exists $x \in \closure\{S\}$
such that $\lim_{k \to \infty}  x_k = x$. Then $y = \lim_{k \to \infty} (\frac{1}{1+ a\tr x_k}, \frac{x_k\tr}{1+ a\tr x_k})\tr
= (\frac{1}{1+ a\tr x}, \frac{x\tr}{1+ a\tr x}) = \ns{x}$.
\end{proof}

Notice that Lemma~\ref{lem:tilde+}\eqref{it:2}, allows to write $\tilde{h}$ in terms of $\bar{h}$ for any
$h \in \R_d[x]$ with $\deg(h) = d$. Also,
Lemma~\ref{lem:condition} allows to write $\rec S$ in terms of $\ns{S}$
for any set $S \subseteq K$.
Lemma~\ref{lem:zzero} below,
characterizes cases in which~\eqref{eq.condition}
can be equivalently formulated using the maps~\eqref{eq:popbar} and~\eqref{eq:setbar}.

\begin{lemma}
\label{lem:zzero}
Let $S \subseteq K$ be closed and
$(x_0, x) \in \ns{S}$ be such that $x_0 > 0$. Then $(x_0, x) \in \ns{S \cap h^{-1}(0)}$ if and only if $(x_0, x) \in \ns{S}\cap \ns{h}^{-1}(0)$.
\end{lemma}

\begin{proof}
Let $y:=(x_0,x) \in \ns{S \cap h^{-1}(0)}$. From Lemma~\ref{lem:z0pos} it follows that $y = \ns{x}$ with $x \in S \cap h^{-1}(0) \subseteq S$.
Thus $y \in \ns{S}$. Also, from $x \in h^{-1}(0)$ and  Lemma~\ref{lem:tilde+}\eqref{it:4} it follows that
$0 = h(x) = \ns{h}(\ns{x})(1 + a\tr x)^d = \ns{h}(y)(1 + a\tr x)^d$. Since for any $x \in S$, $(1+ a\tr x) > 0$ (from~\eqref{eq:adef}),
then $\ns{h}(y) = 0$. Thus $y \in \ns{S}\cap \ns{h}^{-1}(0)$. For the other direction, now let $y \in \ns{S}\cap \ns{h}^{-1}(0)$. Since $y \in \ns{S}$ from
Lemma~\ref{lem:z0pos} it follows that $y = \ns{x}$ with $x \in S$. Also, from $y \in \ns{h}^{-1}(0)$ and  Lemma~\ref{lem:tilde+}\eqref{it:4} it follows that
$0 = \ns{h}(y) = \ns{h}(\ns{x}) = h(x)(1 + a\tr x)^{-d} = h(x)x_0^d$. Since $x_0 > 0$
then $h(x) = 0$. Thus $y \in \ns{S \cap h^{-1}(0)}$.
\end{proof}

Now we are ready to present the two main building blocks for the proof of Theorem~\ref{thm:finerEqGen}.

\begin{theorem}\label{thm:finerEq}
Let $S\subseteq K$ be a closed set and $h \in \P_d(S)$ be of degree $d$.
Then,
\begin{equation}
\label{eq:finer}
\closure(\P_d(S) + h(x)\R)= \{p(x) \in \P_d(S\cap h^{-1}(0)): \tilde p(x) \in \P_d( \rec{S} \cap \tilde h^{-1}(0))\}.
\end{equation}
\end{theorem}

\begin{proof}
Let $p(x) \in \closure(\P_d(S) + h(x)\R)$. By Lemma~\ref{lem:subset}, $p(x) \in \P_d( S \cap h^{-1}(0))$.
From Lemma~\ref{lem:barclosure} we also have $\ns p(y) \in \P_d( \ns{S} \cap \ns{h}^{-1}(0)).$
Given $x \in \rec{S} \cap {\tilde h}^{-1}(0)$, let $y := (0,\frac x{a\tr x})$. From Lemma~\ref{lem:condition} we have $y \in \ns{S}$. Also, from Lemma~\ref{lem:tilde+}\eqref{it:2} it follows that $\ns{h}(y) = \tilde h(\frac x{a\tr x}) = \tilde h(x)(a\tr x)^{-\deg(h)} =0$.
Then $y \in \ns{S} \cap \ns{h}^{-1}(0)$ and thus $0 \le \ns p(y) = \tilde p(\frac x{a\tr x}) = \tilde p(x)(a\tr x)^{-\deg(p)}$ which implies $\tilde p(x) \ge 0$. For the other inclusion,
let $p(x) \in \P_d( S \cap h^{-1}(0))$ be such that $\tilde p(x) \in \P_d( \rec S \cap \tilde h^{-1}(0))$. From Lemma~\ref{lem:barclosure} is enough to show that $\ns p(y) \in \P_d( \ns{S} \cap \ns{h}^{-1}(0))$. Now given $y:=(x_0,x) \in \ns{S} \cap \ns{h}^{-1}(0)$ we consider two cases. If $x_0 > 0$, then by Lemma~\ref{lem:zzero} we have $y \in \ns{S \cap h^{-1}(0)}$. Thus, by Lemma~\ref{lem:tilde+}\eqref{it:4} it follows that $\ns p(y) \ge 0$. If $x_0 = 0$, by Lemma~\ref{lem:condition}
we have $y = (0,x)$ where $x \in \rec S$ and $a\tr x =1$. Using Lemma~\ref{lem:tilde+}\eqref{it:2} we have $\tilde h(x) = \ns h(y) = 0$, that is $x \in \rec S \cap \tilde h^{-1}(0)$. Then, applying Lemma~\ref{lem:tilde+}\eqref{it:2} again, we have $\ns p(y) = \tilde p(x)  \ge 0$.
\end{proof}

\begin{lemma}\label{lem:jaja}
Let $S \subseteq K$ be a closed set, and $a \in \R^n$ be defined by \eqref{eq:adef}. Then for all $d \in \N$,
\begin{enumerate}[(i)]
\item $\int \P_{d}(S) - \int \P_{d}(S) = \R_d[x]$,
\label{lem:jajaint2}
\item $\int \P_{d}(S) \neq \emptyset$. In particular, $g(x):= (1+a\tr x)^d \in \int \P_{d}(S)$, and $\tilde{g}(x) > 0$ for all $x \in K\setminus\{0\}$.
\label{lem:jajaint1}
\end{enumerate}
\end{lemma}
\begin{proof}

Clearly, $\int \P_{d}(S) - \int \P_{d}(S) \subseteq \R_d[x]$. We need to show $\int \P_{d}(S) - \int \P_{d}(S) \supseteq \R_d[x]$, but we claim it follows from~\eqref{lem:jajaint1}.
To see this, take $q(x) \in \int \P_{d}(S)$. For any $p(x) \in \R_d[x]$, there exists $\epsilon >0$
such that $q(x) - \epsilon p(x) \in \int \P_{d}(S)$. Thus $p(x) = \frac{1}{\epsilon} q(x) - (\frac{1}{\epsilon} q(x) - p(x)) \in \int \P_{d}(S) - \int \P_{d}(S)$.

 To prove \eqref{lem:jajaint1}, we show that
$g(x):=(1+ a\tr x)^d \in \int \P_{d}(S)$, where $a$ is defined by \eqref{eq:adef}. Let $\Gamma_d = \{\beta \in \N^n:\|\beta\|_1 \le d\}$, given $p(x) \in \R_d[x]$, write
$p(x) = \sum_{\beta \in \Gamma_d} p_\beta {\|\beta\|_1 \choose \beta} x^\beta$, where for any $x \in \R^n$, $x^\beta := x_1^{\beta_1} \cdots x_n^{\beta_n}$.
Then, for any $x \in S$ we have
\begin{align*}
p(x) &\le |p(x)| \le \|p\|_\infty \sum_{\beta \in \Gamma_d} {\|\beta\|_1 \choose \beta} |x|^\beta = \|p\|_\infty \sum_{k=0}^d(e\tr |x|)^k \\
&\le (d+1)  \|p\|_\infty  \max \{1, (\|x\|_1)^d\}
\le (d+1) \|p\|_\infty \max\{1, n^{d/2}\|x\|^d\},
\end{align*}
 where  for any $x \in \R^n$, $|x| :=(|x_1|, \dots, |x_n|)\tr$.
On the other hand, from Lemma~\ref{lem:alpha} there is $\alpha >0$ such that $a\tr x \ge \alpha \|x\|$ for all $x \in K$.
Thus
\[g(x) = (1+a\tr x)^d \ge (1+\alpha \|x\|)^d > \max\{1, (\alpha \|x\|)^d\}
\]
 for all $x \in K$.
Let $|\epsilon |\le \frac 1{(d+1) \|p\|_\infty} \min\{ 1, (\alpha n^{-1/2})^d\}$, we have then that $g(x)+ \epsilon p(x) \in \P_{d}(S)$. Finally, notice that from~\eqref{eq:adef}, it follows that $\tilde{g}(x) = (a\tr x)^d > 0$ for all $x \in K\setminus \{0\}$.
\end{proof}

Using Theorem~\ref{thm:finerEq} and Lemma~\ref{lem:jaja}, we can now
prove Theorem~\ref{thm:finerEqGen} in Section~\ref{sec:main.result} as follows.

\begin{proof}[Proof of Theorem~\ref{thm:finerEqGen}] Let $d':=d-\deg(h)$. We claim that
\[\closure(\P_d(S) + h(x)(\int\P_{d'}(S) - \int\P_{d'}(S)))= \{p(x) \in \P_d(S\cap h^{-1}(0)): \tilde p(x) \in \P_d( \rec{S} \cap \tilde h^{-1}(0))\},\]
and then the result follows from Lemma~\ref{lem:jaja}\eqref{lem:jajaint2}. Now to show the claim,
we use Lemma~\ref{lem:jaja}\eqref{lem:jajaint1}. Let $g(x) = (1+a\tr x)^{d-\deg(h)} \in \int \P_{d-\deg(h)}(S)$, and $h_1(x) = g(x)h(x)$. Thus $h_1(x) \in P_d(S)$ and $\deg(h_1) = d$. From Theorem~\ref{thm:finerEq}, we have
\begin{align}
\label{eq:gdef}
\closure(\P_d(S) + h_1(x)\R)
&= \{p(x) \in \P_d(S\cap h_1^{-1}(0)): \tilde p(x) \in \P_d( \rec{S} \cap \tilde h_1^{-1}(0))\}\\
&= \{p(x) \in \P_d(S\cap h^{-1}(0)): \tilde p(x) \in \P_d( \rec{S} \cap \tilde h^{-1}(0))\},
\end{align}
where the last equality follows from  Lemma~\ref{lem:jaja}\eqref{lem:jajaint1}, which implies that $h_1^{-1}(0) = h^{-1}(0)$, and $\tilde{h}_1^{-1}(0) = \tilde{h}^{-1}(0)$. Similary, for any $g(x) \in \int \P_{d-\deg(h)}(S)$, and letting $h_1(x) = g(x)h(x)$. one has that
\begin{align}
\label{eq:ggen}
\closure(\P_d(S) + h_1(x)\R)
&= \{p(x) \in \P_d(S\cap h_1^{-1}(0)): \tilde p(x) \in \P_d( \rec{S} \cap \tilde h_1^{-1}(0))\}\\
& \subseteq \{p(x) \in \P_d(S\cap h^{-1}(0)): \tilde p(x) \in \P_d( \rec{S} \cap \tilde h^{-1}(0))\},
\end{align}
where the last inclusion follows from $h_1^{-1}(0) = h^{-1}(0)$ and $\tilde{h}_1^{-1}(0) \supseteq \tilde{h}^{-1}(0)$.
From~\eqref{eq:gdef}, and~\eqref{eq:ggen}, it follows that 
\[
\begin{array}{lcl}
\{p(x) \in \P_d(S\cap h^{-1}(0)): \\
\qquad \qquad \qquad \tilde p(x) \in \P_d( \rec{S} \cap \tilde h^{-1}(0))\}
& = &\displaystyle \bigcup_{g(x) \in \int\P_{d'}(S)} \closure(\P_d(S) + h(x)g(x)\R)\\
& =  &\closure(\P_d(S) + h(x)(\int \P_{d'}(S) - \int\P_{d'}(S))),
\end{array}
\]
where the last equality follows from the fact that the union above is closed.
\end{proof}

\section{A generic counterexample}
\label{sec:counter}
We next show that indeed the statement of Theorem~\ref{main.thm.equ} fails if condition
 \eqref{eq.condition} is violated.
Let $a \in \R^n$ be defined by \eqref{eq:adef}.  Assume condition~\eqref{eq.condition}
in Theorem~\ref{main.thm.equ} does not hold.
Thus, there exists $s \in \rec{S} \cap \tilde{h}^{-1}(0)$ such that $s \not \in \rec{(S \cap h^{-1}(0))}$. Considering $s(a\tr s)^{-1}$ we assume $a\tr s =1$. From Lemma~\ref{lem:condition} it
follows that $(0,s) \in \ns{S}$. Also, by Lemma~\ref{lem:tilde+}\eqref{it:2}, $\ns{h}(0,s) = \tilde{h}(s) = 0$.
Therefore, $(0,s) \in \ns{S} \cap \ns{h}^{-1}(0)$. On the other hand, using Lemma~\ref{lem:condition} again, $(0,s) \not \in \ns{S \cap h^{-1}(0)}$. From Lemma~\ref{lem:barCompact}, $\ns{S \cap h^{-1}(0)}$ is compact. Thus $\epsilon =
\inf\{\|y - (0,s)\|:y \in \overline{S\cap h^{-1}(0)}\}  > 0.$

For any $d \ge 2$, take
\[
p(x):= (1+a\tr x)^{d-2}\left(1+\|(x - (1+a\tr x) s\|^2 - \epsilon^2(1+a\tr x)^2\right).
\]
We claim that $p(x) \in \P_d(S\cap h^{-1}(0))$ but $p(x) \not \in \closure(\P_d(S) + h(x) \R_{d-\deg(h)}[x])$.
Let $x \in S\cap h^{-1}(0)$, using Lemma~\ref{lem:tilde+}\eqref{it:3} we have
\[
\ns{p}(\ns{x}) = \|\ns{x}-(0,s)\|^2 - \epsilon^2 \ge 0.
\]
Thus, $\ns{p}(y) \in \P_d(\ns{S\cap h^{-1}(0)})$ and by Lemma~\ref{lem:tilde+}\eqref{it:4}, $p(x) \in \P_d(S\cap h^{-1}(0))$.
On the other hand,
$
\ns{p}(0,s) = -\epsilon^2.
$  Hence,
there exists $\delta>0$ such that for all $q \in \R_d[x]$
\begin{equation}\label{cont}
\|q-p\|<\delta \Rightarrow \ns{q}(0,s) < -\epsilon^2/2.
\end{equation}
Now to show that $p(x) \not \in \closure(\P_d(S) + h(x) \R_{d-\deg(h)}[x])$ we proceed by contradiction. Assume
 there exist $r_1(x) \in \P_d(S)$ and $r_2(x) \in\R_{d-\deg(h)}[x]$ such that $\|r_1 + h r_2 - p\|<\delta$. From \eqref{cont} we get
\begin{equation}\label{contra1}
\ns{r}_1(0,s) + \ns{h}(0,s)  \ns{r}_2(0,s) < -\epsilon^2/2.
\end{equation}
But, this is a contradiction because $\ns{h}(0,s) = 0$, and  $\ns{r}_1(0,s) \ge 0$ since $(0,s) \in \ns{S}$ and
from Lemma~\ref{lem:tilde+}\eqref{it:4}, $\ns{r}_1 \in \P_d(\ns{S})$.

\section{Horizon cone condition}
\label{sec:infconds}
In this section we provide some further details regarding the horizon cone (cf., Definition~\ref{def:horizon})
conditions used throughout the article. First, notice that condition~\eqref{eq.condition} generically holds in
one direction.

\begin{proposition}\label{prop:oneSideFree} For any $S \subseteq \R^n$ and $h \in \R[x]$ we have $\rec{(S \cap h^{-1}(0))} \subseteq \rec{S} \cap \tilde h^{-1}(0).$
\end{proposition}
\begin{proof}
Let $d = \deg(h)$, and assume $ y \in \rec{(S \cap h^{-1}(0))}$.  Then there are sequences $x^k \in S, \lambda^k\in \R_+,\; k=1,\dots$  such that $h(x^k) = 0$, $\lambda^k\downarrow 0$ and $\lambda^k x^k\rightarrow y$. Thus, in particular, $ y \in \rec{S}$. On the other hand, for $\ell < d$ let $f_\ell(x)$ be the homogeneous component of $h(x)$ of degree $\ell$. We have that
\[
\tilde h(y)
= \lim_{k\rightarrow \infty}  (\lambda^k)^d \tilde h( x^k)
= \lim_{k\rightarrow \infty} ( \lambda^k)^d \brac{ h(x^k)- \sum_{\ell< d} f_{\ell}(x^k)}
= \lim_{k\rightarrow \infty} \sum_{\ell< d}(\lambda^k)^{d-\ell}f_\ell(\lambda^k x^k) = 0.
\]
\end{proof}

Next we examine a case in which the horizon cone condition~\eqref{eq.condition} holds
used in the article.

\begin{proposition}
\label{prop:infconds}
Let $q(x,x_{m+1}) = (x_{m+1} - \frac{1}{4}) - \|x\|^2$, let $h(x,x_{m+1})$ be a non-negative quadratic polynomial, and let $S = \{(x,x_{m+1}) \in \R^{m+1}: q(x,x_{m+1}) \ge 0 \}$. If $S \cap h^{-1}(0) \neq \emptyset$
then $\rec{(S \cap h^{-1}(0))} = \rec{S} \cap \tilde h^{-1}(0)$.
\end{proposition}

\begin{proof}
The inclusion $\rec{( S \cap h^{-1}(0))} \subseteq \rec{S} \cap \tilde{h}^{-1}(0)$ generically holds as stated in Proposition~\ref{prop:oneSideFree} above.
Now we show $\rec{( S \cap h^{-1}(0))} \supseteq \rec{S} \cap \tilde{h}^{-1}(0)$.
To do this, notice that if $(y,y_{m+1}) \in \rec{S}$, then $(y,y_{m+1})  = \lim_{k \to \infty} \lambda_k (x^k,x^k_{m+1})$ with
$\lambda_k \downarrow 0$, and $x_{m+1}^k \ge \frac 14 + \|x^k\|^2$. Thus
$y_{m+1} \ge \lim_{k \to \infty} \lambda_k (\frac 14 + \|x^k\|^2) \ge 0$ and $\|y\|^2 = \lim_{k \to \infty} \lambda_k^2\|x^k\|^2 \le \lim_{k \to \infty} \lambda_k^2(x^k_{m+1} - \frac 14) = 0$, which implies $y =0$.
Therefore $\rec S \subseteq  \{0\}^n \times  \R_+$.

Let $(y,y_{m+1}) \in \rec{S} \cap \tilde h^{-1}(0)$. Then $y=0$ and $y_{m+1} \ge 0$.
Notice that from the non-negativity of $h$ it follows that $h(x,x_{m+1}) = \|Ax + x_{m+1} a - b\|^2$,  where $A \in \R^{m\times n}$, $a\in \R^n$ and $b \in \R^n$ for some $n>0$.
We have then $0 = \tilde h(y,y_{m+1}) = \|Ay + y_{m+1} a\|^2 = y_{m+1}^2\| a\|^2$.

 Now we consider two cases. If $a \neq 0$, we obtain $y_{m+1} = 0$ and thus $(y,y_{m+1})= (0,0) \in \rec{( S \cap h^{-1}(0))}$.  If $a = 0$, let $(\hat x, \hat x_{m+1}) \in S \cap h^{-1}(0)$. Define $\lambda_k = \frac{y_{m+1}}{k\|\hat x_k\|^2 + \sfrac 14}$, $x^k = \hat x$, and $x^k_{m+1} = k \|\hat x\|^2 + \frac 14$,  for
$k\ge 1$. Note that $q(x^k,x^k_{m+1}) = (k-1)\|\hat x\|^2 \ge 0$ for $k \ge 1$, and that $\lim_{k \to \infty} \lambda_k x^k = 0 =y $,
and $\lim_{k \to \infty} \lambda_k x_{m+1}^k = y_{m+1}$. Also, $h(x^k,x^k_{m+1}) = \|Ax^k - b\|^2 = \|A\hat x - b \|^2 = 0$.
Thus, $(y,y_{m+1}) = (0,y_{m+1}) \in \rec{( S \cap h^{-1}(0))}$.
\end{proof}

The following is an example of a case in which the horizon cone condition in Theorem~\ref{thm:ineq} holds
that is used in Proposition~\ref{prop:apxnetzer}.

\begin{proposition}
\label{prop:infcond3}
Let $S = \R^2_+$ and $h(x_1,x_2) = (x_2-x_1^2)(2x_1^2-x_2)$. Then $\rec{S} \cap \tilde h^{-1}(\R_+) = \rec{(S \cap h^{-1}(\R_+))}$.
\end{proposition}

\begin{proof}
Let $U = S \cap h^{-1}(\R_+)$.
On one hand,
$
\rec{S} \cap \tilde h^{-1}(\R_+) = \rec{(\R^2_+)} \cap \{(x_1,x_2): -2x_1^4 \ge 0\} = R^2_+ \cap \{ (0,x_2): x_2 \in \R\} = \{(0,x_2): x_2 \ge 0\}$.
On the other hand,
for any $x_2\ge0$, let $x^k =(k, k^2)$, and  $\lambda_k = \frac{x_2}{k^2}$ for $k=1,\dots$.
Note that $x^k\in U, k=1,\dots.$, and $\lim_{k \to \infty} \lambda_k x^k
= (0,x_2)$. Thus $\rec{U} \supseteq \{(0,x_2): x_2 \ge 0\}$. Furthermore,
let $\lambda_k$, $x^k$, $k=1,2,\dots$ be a sequence associated
to $y \in \rec{U}$. Since $x^k\in U$, $k=1,2,\dots$ we have that
$(x_1^k)^2 \le x_2^k \le 2(x_1^k)^2$, $k=1,2,\dots$. From $\lambda_k > 0$, $k=1,2,\dots$,
it then follows that:
\[
\begin{array}{rcccll}
\dlim_{k \to \infty} \lambda_k^2 (x_1^k)^2 & \le & \dlim_{k \to \infty}  \lambda_k^2 x_2^k & \le &  \dlim_{k \to \infty}  \lambda_k^2 2 (x_1^k)^2,\\[1ex]
\dlim_{k \to \infty} (\lambda_k x_1^k)^2 & \le&  \dlim_{k \to \infty}  \lambda_k (\lambda_k x_2^k) & \le & \dlim_{k \to \infty}  2 (\lambda_k x_1^k)^2,\\[1ex]
y_1^2 & \le & 0 & \le & 2y_1^2,\\
\end{array}
\]
which implies  $\rec{U} \subseteq \{(0,x_2): x_2 \ge 0\}$.
Thus, $\rec{S} \cap \tilde h^{-1}(\R_+) = \{(0,x_2): x_2 \ge 0\} = \rec{(S \cap h^{-1}(\R_+))}$.
\end{proof}

\section{Concluding remarks}
\label{sec:remarks}

The work of \citet{Schm91, Puti93, Lass01} among many others, shows that
the question of when the non-negativity of a polynomial on a semialgebraic set can be
characterized via the associated quadratic module or pre-order is a central question
in the literature. Here we have considered a less restrictive form of this question
in which the relationship between a set $S$ and a polynomial $h$, together with
the non-negative polynomials in $S$, can be used to characterize the non-negativity
of polynomials in the set $S$ intersected with the zeros of $h$ without requiring compactness
assumptions on the set $S$. This result draws an interesting
parallel between results like Putinar's Positivstellensatz and the $S$-Lemma
as follows: One one hand,  Putinar's Positivstellensatz can be used
to write a hierarchy of LMI approximations for a very general class
of PO problems. On the other hand, the $S$-Lemma gives
a LMI reformulation for  a specific class of quadratic PO problems.
The fact that  the $S$-Lemma provides a LMI reformulation
of the problem (instead of a hierarchy of LMI approximations) can be
seen as a consequence of knowing the degree of the polynomials
involved in the certificate of non-negativity, as opposed to Putinar's Positivstellensatz
where the degree of the polynomials used to certify non-negativity are
not known a priori. In this context, the results presented here
provide an interesting bridge between the results on
certificates of non-negativity in algebraic geometry and
 results on certificates of non-negativity arising in
 the general area of quadratic programming. As shown throughout the
 article, this can be used to obtain novel results regarding the
 characterization of non-negative polynomials on possibly unbounded sets.

 We believe that further study of the characterization provided in Theorem~\ref{thm:finerEqGen}
 can lead to a wider application of polynomial optimization techniques for problems with
 unbounded feasible sets. Also, in the related literature both LMI hierarchies based on
 semidefinite programming and
 second-order cone programming  have become the most popular classes of
 LMI hierarchies in the area of polynomial optimization. The
 results presented here regarding linear programming hierarchies
 for polynomial optimization problems motivate further study of this type of relaxations.

%
%
%

\section*{Acknowledgments.}
The authors would like to thank Markus Schweighofer for his valuable comments on a preliminary
version of this manuscript. The work of the first author is supported by NSF grant CMMI-1534850.
The work of the third author is supported by NSF grant CMMI-1300193.




\end{document}